\documentclass[11pt,reqno]{article}
\usepackage{geometry}
\usepackage[noblocks]{authblk}
\usepackage[bitstream-charter]{mathdesign}
\usepackage[T1]{fontenc}
\usepackage{url} 
\usepackage{amsmath,amsthm,mathrsfs}
\usepackage{latexsym}

\usepackage{lipsum}
\usepackage[mathscr]{eucal}
\usepackage[mathlines]{lineno}
\usepackage{latexsym}
\usepackage{xstring}
\usepackage{graphicx,booktabs,multirow}
\usepackage{tikz}
\usetikzlibrary{bbox}
\usepackage{appendix}
\usepackage{yhmath}
\usepackage{float}
\usepackage{color}
\usepackage{subfigure}
\usepackage{rotating}
\usepackage[section]{placeins}
\usepackage{setspace}
\usepackage{pdfpages}
\usepackage{indentfirst} 

\newtheorem{theorem}{Theorem}
\newtheorem{lemma}[theorem]{Lemma}
\newtheorem{corollary}[theorem]{Corollary}
\newtheorem{observation}[theorem]{Observation}
\newtheorem{problem}[theorem]{Problem}
\newtheorem{proposition}[theorem]{Proposition}

\newtheorem{claim}{Claim}

\newtheorem{conjecture}[theorem]{Conjecture}

\newtheorem{case}{Case}
\newtheorem{subcase}{Subcase}[case]

\usepackage[colorlinks,
linkcolor=red!60!black,
anchorcolor=blue,
citecolor=blue!60!black
]{hyperref}

\numberwithin{equation}{section}

\usepackage{pdfpages} 
\usepackage{caption}

\usetikzlibrary{backgrounds}
\usetikzlibrary{arrows}
\usetikzlibrary{shapes,shapes.geometric,shapes.misc}
\usetikzlibrary{patterns}

\let\oldbibliography\thebibliography
\renewcommand{\thebibliography}[1]{%
  \oldbibliography{#1}%
  \setlength{\itemsep}{-2pt}%
  \setlength{\baselineskip}{15pt}
  \setlength{\lineskiplimit}{-\maxdimen}
}

\usepackage[colorlinks,
linkcolor=red!60!black,
anchorcolor=blue,
citecolor=blue!60!black
]{hyperref}
\usepackage{url}

\numberwithin{equation}{section}

\usepackage{pdfpages} 
\usepackage{caption}

\usetikzlibrary{backgrounds}
\usetikzlibrary{arrows}
\usetikzlibrary{shapes,shapes.geometric,shapes.misc}
\pgfdeclarelayer{edgelayer}
\pgfdeclarelayer{nodelayer}
\pgfsetlayers{background,edgelayer,nodelayer,main}
\tikzstyle{none}=[inner sep=0mm]
\tikzstyle{blacknode}=[fill=black, draw=black, shape=circle, minimum
size=0.20cm, inner sep=0pt]
\tikzstyle{whitenode}=[fill={rgb,255: red,245; green,245; blue,245},
draw=black, shape=circle, minimum size=0.20cm, inner sep=1pt]
\tikzstyle{whitenode_v1}=[fill={rgb,255: red,245; green,245; blue,245},
draw=black, shape=circle, minimum size=0.4cm, inner sep=0.8pt, scale=1]
\tikzstyle{blacknode_v1}=[fill=black, draw=black, shape=circle, minimum
size=0.1cm, inner sep=0pt]
\tikzstyle{black_bold}=[-, draw=black, line width=0.6mm]
\tikzstyle{blackedge}=[-, draw=black, fill=none, line width=0.3mm]
\tikzstyle{blackedge_thick}=[-, draw=black, line width=0.4mm, fill=none]
\tikzstyle{rededge_thick}=[-, line width=0.45mm, draw=red]
\tikzstyle{blue_thick}=[-, line width=0.5mm, draw=blue]
\tikzstyle{grayedge}=[-,line width=0.2mm, draw={rgb,255: red,64; green,64; blue,64}]

\DeclareUnicodeCharacter{202F}{\,}
\begin{document}

\title{The minimum size of maximal bipartite IC-plane graphs with given connectivity \thanks{The work was supported by the National Natural Science Foundation of China (Grant Nos. 12271157, 12371346)
 and the Postdoctoral Science Foundation of China (Grant No. 2024M760867).}}

\author[a]{Guiping Wang\thanks{wgp@hunnu.edu.cn}}
\author[a]{Yuanqiu Huang \thanks{Corresponding author: hyqq@hunnu.edu.cn}}
\author[b]{Zhangdong Ouyang\thanks{oymath@163.com}}
\author[c]{Licheng Zhang\thanks{lczhangmath@163.com}}
\affil[a]{ College of Mathematics and Statistics, Hunan Normal University, Changsha, 410081, P.R.China}
\affil[b]{College of Mathematics and Statistics, Hunan First Normal University, Changsha, 410205, P.R.China}
\affil[c]{School of Mathematics, Hunan University, Changsha, 410082, China}

\renewcommand*{\Affilfont}{\small\it} 
\date{}
\maketitle

\begin{abstract}

Recently, the problem of establishing bounds on the edge density of 1-planar graphs, including their subclass IC-planar graphs,  has received considerable attention.
 In 2018, Angelini et al. showed that any $n$-vertex bipartite IC-planar graph has at most $2.25n-4$ edges, which implies that bipartite IC-planar graphs have vertex-connectivity at most 4.
In this paper,  we prove that any $n$-vertex maximal bipartite IC-plane graph with connectivity $2$ has at least $\frac{3}{2}n-2$ edges,
and those with connectivity $3$ has at least $2n-3$ edges. All the above lower bounds are tight. For 4-connected maximal bipartite IC-planar graphs, the question of determining a non-trivial lower bound on the size remains open.\\



\end{abstract}

\section{Introduction}

\noindent 
We consider finite simple connected graphs.  Let $G=(V, E)$ be a graph. Its order and size are $|V|$ and $|E|$, respectively.
Throughout, we use $V(G)$ and $E(G)$ to denote the vertex set and edge set of $G$, and write $n(G)$ and $e(G)$ for its order and size, respectively.
The minimum degree of $G$ is denoted by $\delta(G)$.
Let $S$ be a subset of $V(G)$. 
The {\it connectivity} $\kappa(G)$ is the minimum integer $|S|$ such that $G-S$ is disconnected or $n(G- S)=1$.
A graph $G$ is {\it $k$-connected} if $\kappa(G)\ge k$.

A {\it drawing} of a graph $G=(V,E)$ is a mapping $D$ that assigns to each vertex in $V$ a distinct point in the plane 
and to each edge $uv$ in $E$ a continuous arc connecting $D(u)$ and $D(v)$.
We often make no distinction between a graph-theoretical object (such as a vertex, or an edge) and its drawing.
A drawing is {\it good} if no edge crosses itself, no two edges cross more than once, 
and no two edges incident with the same vertex cross each other. 
All drawings considered here are good.
A graph is {\it planar} if it admits a drawing in the plane such that no edges cross each other,
and such a drawing is called a {\it plane} graph.
A drawing of a graph is {\it $1$-planar} if each  of its edges is crossed at most once.
A graph is {\it $1$-planar} if it has a $1$-planar drawing.
A graph together with a $1$-planar drawing is called a {\it 1-plane graph}.
Given a graph family $\mathcal{G}$, $G$ is {\it maximal} in $\mathcal{G}$ if no edge can be added to $G$ without violating the defining class.

\subsection{Background on the density of 1-planar graphs}


\noindent 
Determining the minimum size of an $n$-vertex graph from a given graph class is a fundamental problem in graph theory. 
 Recently, the related study on the density of graphs with fixed drawings has received significant attention \cite{CB2017,MA2024,YQ2021,MH2023}.
Maximal planar graphs are always maximal plane graphs.
It is well-known that every maximal planar graph  with $n$ ($\geq3$) vertices is triangulated and has $3n-6$ edges.
The densest and sparsest maximal planar graphs coincide.
However, this property no longer hold for $1$-planar graphs.
Maximal $1$-planar graphs are always maximal $1$-plane graphs, but not vice versa.
Actually, maximal $1$-planar graphs form a proper subset of maximal $1$-plane graphs.
Many authors independently proved that any (maximal) 1-planar graph on $n$ vertices has at most $4n-8$ edges, 
and this bound is tight for $n=8$ and all $n\geq 10$ \cite{HB1984,IF2007,JP1997}.
Consequently, the connectivity of any $1$-planar graph is at most $7$.
Brandenburg et al.~\cite{FJ2013} constructed sparse maximal $1$-planar graphs with approximately $2.65n$ edges, as well as maximal $1$-plane graphs with about $2.33n$ edges.
On the other hand, they established lower bounds of $\frac{28}{13}n-O(1)$ for the size of maximal 1-planar graphs 
and $2.1n-O(1)$ for the size of maximal 1-plane graphs.
Later on, Barát and Tóth \cite{JB2018} improved both lower bounds to $\frac{20}{9}n-O(1)$.
Recently, Huang et al. \cite{YQ2025} provided a tight lower bound of $\lceil\frac{7}{3}n\rceil-3$,
which is achieved by an infinite family of $2$-connected graphs.
Extending these results, Ouyang et al. \cite{ZD2025} investigated the crossing number and minimum size of $k$-connected maximal $1$-plane graphs for $3\leq k\leq 7$.

These results reveal two interesting phenomena, which were first identified by Brandenburg et al.

\begin{itemize}
  \item  There is a significant gap between the upper and lower bounds for maximal $1$-planar graphs; 
  \item there exist maximal $1$-planar (or $1$-plane) graphs that are even sparser than maximal planar graphs.
\end{itemize}
Interestingly, both phenomena persist even in certain subclasses of 1-planar graphs.
In a $1$-plane graph (even when bipartite), any two pairs of crossing edges share at most three common end-vertices (see Figure \ref{common end}).
A graph is {\it IC-planar} (independent crossing planar) if it has a $1$-planar drawing such that any two pairs of crossing edges share no common end-vertex. 
IC-planar graphs have been extensively studied from recognition algorithm, coloring and structural perspectives. For example,
Brandenburg et al. proved that IC-planarity testing is NP-complete both in general and for a fixed rotation system \cite{FJ2016}, 
while triangulated IC-planar graphs can be recognized in cubic time \cite{FJ2018}.
Král' and Stacho \cite{DK2010} showed that every IC-planar graph has a chromatic number of at most $5$. 
Zhang and Liu \cite{XZ2013} showed that any $n$-vertex IC-planar graph has at most $3.25n-6$ edges,
and every IC-planar graph with the maximum number of edges is the union of a triangulated planar graph and a set of matching edges.
This upper bound is tight, which evidently holds for IC-plane graphs.
For any maximal IC-plane graph $G$, Ding, Huang and Lv \cite{ZP2025} recently showed that $e(G)\geq\frac{7}{3}n-\frac{14}{3}$.

\begin{figure}[h!]
  \centering
\begin{tikzpicture}[scale=0.5]
	\begin{pgfonlayer}{nodelayer}
		\node [style=whitenode] (0) at (2.5, -18) {};
		\node [style=whitenode] (1) at (6.5, -22) {};
		\node [style=blacknode] (2) at (6.5, -18) {};
		\node [style=blacknode] (3) at (2.5, -22) {};
		\node [style=whitenode] (4) at (0.5, -16) {};
		\node [style=whitenode] (5) at (8.5, -24) {};
		\node [style=blacknode] (6) at (8.5, -16) {};
		\node [style=blacknode] (7) at (0.5, -24) {};
		\node [style=blacknode] (8) at (5.325, -21) {};
		\node [style=whitenode] (9) at (5.325, -19) {};
		\node [style=whitenode] (10) at (4.975, -20.05) {};
		\node [style=none] (11) at (4.5, -16.75) {$c_{1}$};
		\node [style=none] (12) at (7.8, -20) {$c_{2}$};
		\node [style=none] (13) at (4.5, -23.25) {$c_{3}$};
		\node [style=none] (14) at (1.25, -20) {$c_{4}$};
		\node [style=none] (15) at (3.575, -20) {$c_{5}$};
		\node [style=none] (16) at (6, -20) {$c_{6}$};
	\end{pgfonlayer}
	\begin{pgfonlayer}{edgelayer}
		\draw [style={blackedge_thick}] (0) to (2);
		\draw [style={blackedge_thick}] (2) to (1);
		\draw [style={blackedge_thick}] (3) to (1);
		\draw [style={blackedge_thick}] (0) to (3);
		\draw [style={blackedge_thick}] (4) to (2);
		\draw [style={blackedge_thick}] (6) to (0);
		\draw [style={blackedge_thick}] (0) to (7);
		\draw [style={blackedge_thick}] (4) to (3);
		\draw [style={blackedge_thick}] (7) to (1);
		\draw [style={blackedge_thick}] (3) to (5);
		\draw [style={blackedge_thick}] (6) to (1);
		\draw [style={blackedge_thick}] (2) to (5);
		\draw [style={blackedge_thick}] (4) to (7);
		\draw [style={blackedge_thick}] (4) to (6);
		\draw [style={blackedge_thick}] (6) to (5);
		\draw [style={blackedge_thick}] (0) to (8);
		\draw [style={blackedge_thick}] (9) to (3);
		\draw [style={blackedge_thick}] (9) to (8);
		\draw [style={blackedge_thick}, in=15, out=-30, looseness=2.25] (10) to (3);
		\draw [style={blackedge_thick}] (10) to (8);
		\draw [style={blackedge_thick}] (7) to (5);
	\end{pgfonlayer}
\end{tikzpicture}

  \caption{A biparite 1-plane graph with six crossings, where $c_{1}$ and $c_{3}$ share no common end, $c_{1}$ and $c_{5}$ share one common end, 
  $c_{1}$ and $c_{2}$ share two common ends, and $c_{5}$ and $c_{6}$ share three common ends. }
  \label{common end}
\end{figure}
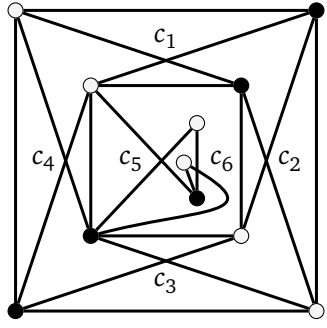

For  1-planar graphs, restricted to bipartite graphs, Karpov \cite{DK2014} showed that every $n$-vertex bipartite $1$-planar graph has at most $3n-8$ edges
and this upper bound is determined by its order.
When the sizes of the partite sets are specified,
Huang et al. \cite{YQ2021} proved that a bipartite 1-planar graph $G$ with partite sets of sizes $x$ and $y$, where $2\leq x \leq y$, has at most $ 2n(G)+4x-12$ edges.


A graph is {\it bipartite IC-plane} if it is bipartite and IC-plane.
 Angelini et al. \cite{PA2018} obtained a tight upper bound on the size of bipartite IC-planar graphs as follows. 
\begin{theorem}[ Angelini et al. \cite{PA2018}]\label{BIC}
Let $G$ be an $n$-vertex bipartite IC-planar graph, then $e(G)\leq\frac{9}{4}n-4$, and the bound is tight.
\end{theorem}

A bipartite IC-plane graph is \emph{maximal} if no edge can be added such that the resulting graph remains bipartite IC-plane and simple. 
For brevity, we refer to a maximal bipartite IC-plane graph as a MBICP-graph.
Motivated by the type of problem initiated by Brandenburg et al.~\cite{FJ2013}, 
it is natural to consider the problem of determining the minimum size of maximal bipartite IC-plane graphs with a given order. 
However, any star graph of order $n$, which is a maximal bipartite IC-plane graph with connectivity one, already attains the trivial lower bound of $n-1$. This observation suggests that the problem becomes more interesting when restricted to 2-connected graphs.
\begin{problem}\label{prob:1}
Given a fixed order and connectivity of at least two, what is the minimum size of a MBICP-graph?
\end{problem}

\subsection{Our contribution}

\noindent The main results of this paper are characterized by the following two theorems.

\begin{theorem}\label{2-connected}

Let $G$ be an $n$-vertex MBICP-graph. 
If $\kappa(G)\ge 2$, then $e(G)\geq \frac{3}{2}n-2$. 
Moreover, this bound is tight for all integers $n=4k$ with $k\geq1$.
\end{theorem}

\begin{theorem}\label{3-connected}
Let $G$ be an $n$-vertex MBICP-graph. 
If $\kappa(G)\ge 3$, then $e(G)\geq 2n-3$.
Moreover, this bound is tight for all integers $n=2k$ with $k\geq3$.
\end{theorem}

Furthermore, we construct a family of $4$-connected MBICP-graphs $G$ with size $\frac{13}{6}n(G)-\frac{14}{3}$,
and propose a conjecture on the lower bounds for the size of 4-connected  MBICP-graphs in Section~\ref{further study}.

 \subsection{Terminology and notations}\label{Terminology}

\noindent In this subsection, we give some  terminology and notations. We use standard terminology and notation following in \cite{JB2008}.

Let $G$ be a graph and $A$ be a subset of $V(G)$.
The {\it induced subgraph} $G[A]$ of $G$ denotes the subgraph with vertex set $A$ and edge set $\{uv\in E(G):u,v\in A\}$.
Let  $G- A$ denote the graph $G[V(G)\backslash A]$.
If $A=\{u\}$, we write $G-u$ instead of $G-\{u\}$. 
If $e$=$uv$ is an edge of $G$, then we say $u$ and $v$ are {\it adjacent}, and $e$ and $u$ are {\it incident}.  We denote by $G\cong H$ if two graphs $G$ and $H$ are isomorphic.

For a given drawing $D$, let $cr(D)$ denote the number of crossings in $D$.  
The {\it crossing number} $cr(G)$ is the minimum value of $cr(D)$ over all drawings $D$ of $G$.

Let $G$ be $1$-plane graph.
An edge is a {\it crossing edge} if it is crossed by another edge in $G$; otherwise, it is a {\it clean edge}.
A cycle $C$ of $G$ is {\it clean} if each edge on $C$ is clean. 
Let $G_{p}$ denote the plane graph obtained from $G$ by arbitrarily removing one edge from each pair of crossing edges. 
The {\it associated plane graph} $G^{\times}$ is the plane graph derived from $G$ by turning all crossings of $G$ into new vertices of degree 4; 
these new vertices are called {\it false} vertices, other vertices of $G^{\times}$ are {\it true} ones. 
The $1$-plane graph $G$ partitions the plane $\mathbb{R}^{2}$ into regions called {\it faces} of $G$,
each bounded by a closed walk consists of vertices, false vertices, edges and edge segments.
Let $\partial(F)$ denote the boundary of a face $F$ and $T(F)$ denote the set of true vertices on $\partial(F)$.
In $G$, a face is {\it false} if its boundary contains at least one crossing; and is {\it true} otherwise.

The {\it size} $|F|$ of a face $F$ is defined as the number of edges and edge segments on its boundary, where each repeated edges or edge segments  is counted twice. 
A face of size $k$ is said to be a {\it k-face},
A graph is a {\it quadrangulation} if it is a simple plane graph with each face of size $4$.
Two faces in a plane graph are {\it adjacent} if their boundaries share at least one common edge.
For any vertex $x$ in $G$, let $N_{G}(x)=\{y\in V(G): xy\in E(G)\}$.
If $\alpha$ is the crossing of edges $e$=$ac$ and $e'$=$bd$, let $N_{G}(\alpha)=\{a,b,c,d\}$.
A crossing $\alpha$ and a vertex $x$ are called {\it incident} if $x \in N_{G}(\alpha)$.
A non-self-crossed cycle $C$ of $G^{\times}$ partitions the plane into two open regions, 
where the bounded region is called the {\it interior} of $C$, and the unbounded region is called the {\it exterior} of $C$.

For any integers $a$ and $b$, let $[a,b]$ denote the set of integers $i$ with $a\leq i\leq b$.

The rest of the paper is organized as follows.
Section \ref{Preliminaries} introduces  some elementary lemmas.
Section \ref{properties} gives the structural properties of 2-connected MBICP-graphs.
In Section \ref{crossing number0}, we characterize MBICP-graphs without crossings.
In Section \ref{proofmaintheorem}, we provide  proofs of Theorems \ref{2-connected} and \ref{3-connected}.

 \section{Preliminaries}\label{Preliminaries}

\noindent The following are two known lemmas.

\begin{lemma}[\cite{XZ2013}]\label{lem:zhang}
Let $G$ be an IC-plane graph and $D$ be the corresponding IC-planar drawing of $G$. Then $cr(D)\leq \frac{n(G)}{4}$.
\end{lemma}

A graph is {\it outer $1$-planar} if it admits a $1$-planar drawing in the plane 
such that all vertices can be drawn on the boundary of a face.
Interestingly, Auer et al. showed that every outer $1$-planar graph is planar.

\begin{lemma}[Auer et al. \cite{CA2016}]\label{O1P}
Every outer $1$-planar graph is planar.
\end{lemma}

Let $K_{p,q}$ denote the complete bipartite graph with partite sets of sizes $p$ and $q$.
Since $K_{3,3}$ is not planar, by Lemma \ref{O1P}, $K_{3,3}$ is not outer $1$-planar.
Hence, we immediately obtain the following lemma.
 
\begin{lemma}\label{OB1P}
Let $K_{s,t}$ be a complete bipartite graph with partite sets of sizes $s$ and $t$.
If $s\ge 3$ and $t\ge 3$, then $K_{s,t}$ is not outer $1$-planar.
\end{lemma}

\begin{lemma}
\label{2-connected property}
Let $G$ be a  $2$-connected $1$-plane graph.
If $F$ is a face of $G$, then any true vertex   on the boundary walk of $F$  occurs exactly once.
\end{lemma}
\begin{proof}
Suppose that $v\in T(F)$ appears more than once on the  boundary walk of $F$.
A simple Jordan closed curve $\mathcal{O}$ connecting $v$ to $v$ can be drawn within $F$ such that $\mathcal{O}\cap G=\{v\}$ 
and both regions partitioned  by $\mathcal{O}$ contain the vertices of $G$.
This implies that $v$ is a cut-vertex of $G$, contradicting the assumption that $G$ is $2$-connected. 
\end{proof} 

\begin{lemma}\label{lem:face}
Let $G$ be a MBICP-graph and $F$ be a face of $G$.  Then any two vertices with different colors on $\partial(F)$ are adjacent in $G$.
\end{lemma}

\begin{proof}Suppose that $u$ and $v$ are two vertices on $\partial(F)$ with different colors and non-adjacent in $G$.
Then, a clean edge $uv$ can be added within $F$,
contradicting the maximality of $G$.
\end{proof}

We note that a particular 1-planar drawing of $K_{2,2}$, that is, a 4-cycle, will play an important role in what follows.
A {\it tie} is a $1$-plane graph that is isomorphic to the complete graph $K_{2,2}$ with two edges crossing each other, as shown in Figure \ref{tie0}. 
Of course, a tie divides the plane into three regions, and any one of these regions can be designated to project onto the outer infinite region.

\begin{figure}[h!]
  \centering
\begin{tikzpicture}[scale=0.75]
	\begin{pgfonlayer}{nodelayer}
		\node [style=blacknode] (0) at (-2, 2.75) {};
		\node [style=whitenode] (1) at (1.5, 2.75) {};
		\node [style=whitenode] (2) at (1.5, -0.75) {};
		\node [style=blacknode] (3) at (-2, -0.75) {};
	\end{pgfonlayer}
	\begin{pgfonlayer}{edgelayer}
		\draw [style={blackedge_thick}] (1) to (3);
		\draw [style={blackedge_thick}] (3) to (2);
		\draw [style={blackedge_thick}] (0) to (2);
		\draw [style={blackedge_thick}] (0) to (1);
	\end{pgfonlayer}
\end{tikzpicture}
  \caption{A tie}
  \label{tie0}
\end{figure}

For convenience, we denote a tie by $T(\alpha)$, where $\alpha$ is the crossing of the tie, as shown in Figure \ref{tie}(1).
Let $R_1, R_2$, and $R_3$ denote the three regions  partitioned by $T(\alpha)$, 
where $R_1$ and $R_2$ are bounded by one edge and two edge segments, while $R_3$ is bounded by two edges and four edge segments.
For brevity, we call $R_{1}$ and $R_{2}$ the {\it $3$-patches} of $T(\alpha)$ and $R_{3}$ is the {\it $6$-patch} of $T(\alpha)$.

\begin{figure}[h!]
  \centering
\begin{tikzpicture}[scale=0.65, bezier bounding box]

	\begin{pgfonlayer}{nodelayer}
		\node [style=blacknode] (0) at (-5.75, 3.5) {};
		\node [style=whitenode] (1) at (-2.25, 3.5) {};
		\node [style=whitenode] (2) at (-2.25, 0) {};
		\node [style=blacknode] (3) at (-5.75, 0) {};
		\node [style=blacknode] (4) at (2.55, 3.375) {};
		\node [style=whitenode] (5) at (4.975, 3.375) {};
		\node [style=whitenode] (6) at (4.975, 0.875) {};
		\node [style=blacknode] (7) at (2.55, 0.875) {};
		\node [style=none] (8) at (3.75, 1.375) {\color{blue}$R_{2}$};
		\node [style=none] (9) at (-3.85, -2) {$(1)$};
		\node [style=none] (10) at (4, -2) {$(2)$};
		\node [style=none] (12) at (3.825, 3.125) {\color{blue}$R_{1}$};
		\node [style=none] (13) at (5.075, 1.875) {\color{blue}$R_{3}$};
		\node [style=none] (14) at (-4, 2.75) {\color{blue}$R_{1}$};
		\node [style=none] (15) at (-4, 0.75) {\color{blue}$R_{2}$};
		\node [style=none] (16) at (-2.25, 1.75) {\color{blue}$R_{3}$};
		\node [style=none] (19) at (-4.5, 1.75) {$\alpha$};
		\node [style=none] (20) at (3.225, 2.125) {$\alpha$};
	\end{pgfonlayer}
	\begin{pgfonlayer}{edgelayer}
		\draw [style={blackedge_thick}] (1) to (3);
		\draw [style={blackedge_thick}] (3) to (2);
		\draw [style={blackedge_thick}] (0) to (2);
		\draw [style={blackedge_thick}] (5) to (7);
		\draw [style={blackedge_thick}] (7) to (6);
		\draw [style={blackedge_thick}] (4) to (6);
		\draw [style={blackedge_thick}] (0) to (1);
		\draw [style={blackedge_thick}, in=315, out=-135, looseness=6.50] (4) to (5);
	\end{pgfonlayer}
\end{tikzpicture}

  \caption{$(1)$ $R_{3}$ is an unbounded region in $T(\alpha)$; $(2)$ $R_{3}$ is a bounded region in $T(\alpha)$
}
  \label{tie}
\end{figure}
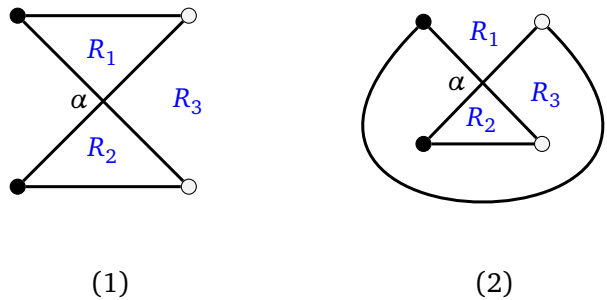

We observe that any pair of crossed edges in a MBICP-graph forms a  tie, as stated in the following lemma.
\begin{lemma}\label{cleantie}
Let $G$ be a MBICP-graph. Let $a b$ and $c d$ be two edges in $G$ that cross at a point $\alpha$, where $\{a,c\}\in X$ and $\{b,d\}\in Y$.
Then $G[\{a,b,c,d\}] \cong T(\alpha)$. Furthermore, edges $ad$ and $cb$ are clean in $G$.
\end{lemma}
\begin{proof}
By the 1-planarity of $G$, neither $ab$ nor $cd$ is crossed by any other edge in $G$.
Then there exists a face $F$ in $G$ whose boundary $\partial(F)$ contains the edge segments $a\alpha$ and $d\alpha$.
By Lemma \ref{lem:face}, $ad\in E(G)$,
Similarly, $bc\in E(G)$.
Thus $G[\{a,b,c,d\}] \cong T(\alpha)$. Furthermore, if either $ad$ or $cb$ is crossed by another edge in $G$, this would contradict the IC-planarity of $G$, and thus the conclusion holds.
\end{proof}

\section{Properties of 2-connected MBICP-graphs}\label{properties}
\noindent 
For simplicity, throughout this section, let $G$ be a $2$-connected MBICP-graph with the vertex partition $(X, Y)$. We color the vertices in $X$ black and the vertices in $Y$ white. This section, we give some structural properties  focusing  on the faces of $G$.

\subsection{Structures of faces of $G$}

\noindent For a face $F$ of $G$, two vertices $u$ and $v$ are called \emph{consecutive }on the boundary $\partial(F)$ if, along the boundary walk of $\partial(F)$, they appear consecutively in order, with no other vertices (except crossings) between them.

\begin{lemma}\label{lem:consecutivevertices}
Let  $F$ be a face of $G$. If $x,y,z$ are three consecutive vertices on   $\partial(F)$, then $x,y,z$ cannot have the same color.
\end{lemma}
\begin{proof} Suppose that $u$, $v$, $w$ are three consecutive black vertices on $\partial(F)$.
Since $G$ is bipartite, there exists a crossing $\alpha_{1}$ between $u$ and $v$, and a crossing $\alpha_{2}$ between $v$ and $w$ on $\partial(F)$.
Then $N_{G}(\alpha_{1})\cap N_{G}(\alpha_{2})=v$, contradicting the IC-planarity of $G$.
\end{proof}

%

%
%
%

%

\begin{lemma}\label{configuration}
Let $ab$ and $cd$ be two edges in $G$ that cross at a point $\alpha$, where $\{a,c\}\subseteq X$ and $\{b,d\}\subseteq Y$.
Let $F$ be a face of $G$ such that $\alpha \in \partial(F)$.
If $\{a,c\}\subseteq \partial(F)$ and $|F|\ge 5$,
then $\{b,d\}\subseteq \partial(F)$, i.e., $N_{G}(\alpha)\cap \partial(F)=\{a,b,c,d\}$.
\end{lemma}
\begin{proof}
Let $ax$ and $cy$ be edges on $\partial(F)$ incident with $a$ and $c$, respectively.
By IC-planarity, both $ax$ and $cy$ are clean.
Since $|F| \geq 5$, it follows that $x \neq y$.
Moreover, as $G$ is simple, we have $d \neq y$ and $b \neq x$.
It suffices to prove $d= x$ and $b= y$. Suppose $d\neq x$ and $b\neq y$.
By the maximality of $G$, $\{ay,cx\}\in E(G)$ and lie in the exterior of $F$.
Consequently, $ay$ and $cx$ cross each other at a point $\alpha'$.
Then $N_{G}(\alpha) \cap N_{G}( \alpha')=\{a,c\}$,
a contradiction to the IC-planarity of $G$.

If $d=x$ and $b\neq y$, then $ay\in E(G)$ and lies in the exterior of $F$.
In this case, $cx$ would be crossed twice, contradicting the $1$-planarity of $G$.
Similarly, the case $d\neq x$ and $b=y$ is also impossible.
Hence $d=x$ and $b=y$. 
Thus $\{b,d\}\subseteq \partial(F)$ and $N_{G}(\alpha)\cap \partial(F)=\{a,b,c,d\}$.
\end{proof}

A graph is {\it outer IC-planar} if it admits an outer $1$-planar drawing in the plane such that two pairs of crossing edges
share no common end-vertex; and such drawing is called \emph{outer IC-plane}.

\begin{lemma}\label{properties of G}
Let  $F$ be a face of $G$. 
Let $\ell_{b}$, $\ell_{w}$, $\ell_{t}$ and $\ell_{f}$ be the numbers of black, white, true and false vertices on $\partial(F)$, respectively.
Then the following statements hold.
\begin{itemize}
  \item[(i)] $G[T(F)]\cong K_{\ell_{b}, \ell_{w}}$ and $G[T(F)]$ is outer $1$-plane; and
  \item[(ii)] if $|\ell_{b}-\ell_{w}|=k$, where $\max \{\ell_{b},\ell_{w}\}\geq2$ and $k\geq1$, 
then $\ell_{f}\geq k$, and moreover, if $\ell_{b}>\ell_{w}$, then there exist at least $k$ crossings $x$ on $\partial(F)$ 
such that two end-vertices of $x$ on $\partial(F)$ are black. Correspondingly, if $\ell_{w}>\ell_{b}$, then there exist at least $k$ crossings $y$ on $\partial(F)$ 
such that two end-vertices of $y$ on $\partial(F)$ are white.
\end{itemize}
\end{lemma}
\begin{proof}

(i). Since $T(F)$ is the set of true vertices on $\partial(F)$, it follows that $\ell_t=\ell_b+\ell_w$. By Lemma \ref{lem:face}, any two vertices of different colors on $\partial(F)$ are adjacent in $G$. Consequently, $G[T(F)] \cong K_{\ell_b, \ell_w}$ holds. Moreover, since $G$ is a MBICP-graph and $F$ is a face of $G$, all edges in $E(G[T(F)])$ except for those in $E(F)$ lie in the exterior of $F$. Thus, $G[T(F)]$ is outer 1-plane.

(ii). Assume $\ell_{b}> \ell_{w}$.
Then $\ell_{b}-\ell_{w}=k$.
The $\ell_{b}$ black vertices partition the boundary of $F$ into $\ell_{b}$ segments $S_{1},S_{2},\dots,S_{\ell_{b}}$.
As $G$ is bipartite, there exists either a white or a false vertex located on each segment $S_{i}$, where $1\leq i\leq \ell_{b}$.
Then $\ell_{w}+\ell_{f}\geq \ell_{b}$. Since $\ell_{b}=\ell_{w}+k$, it follows that $\ell_{f}\geq k$.
Hence, there exist at least $k$ crossings on $\partial(F)$ whose two end-vertices on $\partial(F)$ are black.
\end{proof}

\begin{lemma}\label{lem:facesbound}
Let  $F$ be a face of $G$. If $F$ is true, then $|F|=4$.

\end{lemma}

\begin{proof} Since $G$ is bipartite, by Lemma \ref{2-connected property}, it follows that $F$ is bounded a simple cycle, and so $|F|$ is even and $\ell_{b}=\ell_{w}$. 
By the simplicity of $G$ and $\kappa(G)\ge 2$, we have $|F|\neq2$. If
$|F|\ge 6$, then $\ell_{b}=\ell_{w}=|F|/2\ge 3$. 
By Lemma \ref{properties of G}(i),  $G[T(F)]\cong K_{|F|/2,|F|/2}$. 
By Lemma \ref{OB1P}, $G[T(F)]$ is not outer $1$-planar,
a contradiction to Lemma \ref{properties of G}(i).
Thus $|F|=4$.
\end{proof}

Furthermore, for a face $F$ whose boundary may have  crossings (false vertices), we have the following lemma.

\begin{lemma}\label{true-num}
Let $\ell_{t}$ and $\ell_{f}$ be the numbers of true and false vertices on the boundary of a face $F$ in $G$, respectively.
Then we have
\begin{itemize}
  \item[(i)] $\ell_{f}\leq \frac{1}{2} \ell_{t}$; and
  \item[(ii)]  $2\le \ell_{t} \leq 5$ and $\ell_{f}\leq2$.
\end{itemize}

\end{lemma}
\begin{proof}
(i). By the $1$-planarity of $G$, each false vertex on $\partial(F)$ lies between two true vertices on $\partial(F)$.
Suppose that $\ell_{f}> \frac{1}{2} \ell_{t}$. 
Then, there exists at least one true vertex on $\partial(F)$ incident with two crossings,
a contradiction to that $G$ is IC-plane.
Thus, $(i)$ holds.

(ii). Given that the graph $G$ is 2-connected, and considering its simplicity and the fact that $G$ admits a good drawing, it follows that no face boundary of $G$ contains fewer than two vertices.  Thus we have $\ell_t \geq 2$.

Assume to the contrary that $F$ is a face of $G$ with $\ell_{t}\geq6$.
Let $\ell_b$ and $\ell_w$ denote the numbers of black and white vertices on $\partial(F)$, respectively, and assume without loss of generality that $\ell_b \le \ell_w$.
By Lemma \ref{properties of G} (i), $G[T(F)]\cong K_{\ell_{b}, \ell_{w}}$ and $G[T(F)]$ is outer $1$-plane. 

\begin{claim}
$\ell_{b}= 2$ and $\ell_{w}= 4$.
\end{claim}

\begin{proof}
Suppose $\ell_b = 1$ or $\ell_b = 2$ with $\ell_w \ge 5$. Then there exist at least three consecutive white vertices, contradicting Lemma~\ref{lem:consecutivevertices}. Suppose that $\ell_{b}\ge 3$. Then  $\ell_{w}\ge 3$.  By Lemma \ref{properties of G}(i), $G[T(F)]\supseteq K_{3,3}$, which is not outer 1-planar by Lemma \ref{OB1P}, a contradiction.  Thus the claim holds.
\end{proof}

\begin{claim}
$\ell_{f}=2$.
\end{claim}
\begin{proof}
By (i)  we have $\ell_{f}\leq3$.
By Lemma \ref{properties of G} (ii), 
$\ell_{f}\geq2$, and there are at least two crossings $c_1$ and $c_2$ on $\partial(F)$ whose end-vertices are both white. 
If $\ell_f = 3$, then due to IC-planarity, the two end-vertices of the crossing, other than $c_1$ and $c_2$, on $\partial(F)$ are black.
Consequently, the two black vertices on $\partial(F)$ are consecutive. 
This implies the four white vertices  on $\partial(F)$ are consecutive, a contradiction to Lemma \ref{lem:consecutivevertices}.
Thus $\ell_{f}=2$.
\end{proof}

By the two claims above, we may assume that $\{w_1, w_2, w_3, w_4\}$ is the set of white vertices on $\partial(F)$, $\{b_1, b_2\}$ is the set of black vertices, and $\{\alpha_1, \alpha_2\}$ is the set of crossings on $\partial(F)$.
By Lemma \ref{lem:consecutivevertices}, there are at most two white vertices are consecutive on $\partial(F)$.
Then $F$ is the face depicted in Figure \ref{caseK2,4} $(1)$.
By Lemma \ref{configuration}, we have $c_{1}=b_{1}$, $c_{2}=b_{2}$, $c_{3}=b_{1}$ and $c_{4}=b_{2}$,
as shown in Figure \ref{caseK2,4} $(2)$.
Then $N_{G}(\alpha_{1})\cap N_{G}(\alpha_{2})=\{b_{1},b_{2}\}$, 
a contradiction to IC-planarity of $G$.
Thus $\ell_{t}\leq5$,
and (i) implies $\ell_{f}\leq2$.
Hence, (ii) holds.
\end{proof}

\begin{figure}[h!]
  \centering
\begin{tikzpicture}[scale=0.7]
	\begin{pgfonlayer}{nodelayer}
		\node [style=whitenode] (0) at (-5.75, 3) {};
		\node [style=whitenode] (1) at (-5.75, 0) {};
		\node [style=whitenode] (2) at (-3.25, 3) {};
		\node [style=whitenode] (3) at (-3.25, 0) {};
		\node [style=blacknode] (4) at (-4.5, 3) {};
		\node [style=blacknode] (5) at (-4.5, 0) {};
		\node [style=blacknode] (6) at (-7.5, 0.5) {};
		\node [style=blacknode] (7) at (-7.5, 2.5) {};
		\node [style=blacknode] (8) at (-1.5, 2.5) {};
		\node [style=blacknode] (9) at (-1.5, 0.5) {};
		\node [style=none] (10) at (-4.5, 1.5) {\color{blue}$F$};
		\node [style=whitenode] (11) at (2.25, 2.5) {};
		\node [style=whitenode] (12) at (2.5, 0.5) {};
		\node [style=whitenode] (13) at (5.75, 2.5) {};
		\node [style=whitenode] (14) at (5.75, 0.5) {};
		\node [style=blacknode] (19) at (4, 3) {};
		\node [style=blacknode] (20) at (4, 0) {};
		\node [style=none] (21) at (4, 1.5) {\color{blue}$F$};
		\node [style=none] (22) at (-5.75, 3.5) {$a_{1}$};
		\node [style=none] (23) at (-3.25, 3.5) {$a_{2}$};
		\node [style=none] (24) at (-3.25, -0.5) {$a_{3}$};
		\node [style=none] (25) at (-5.75, -0.5) {$a_{4}$};
		\node [style=none] (26) at (-7.75, 3) {$c_{1}$};
		\node [style=none] (27) at (-7.75, 0) {$c_{2}$};
		\node [style=none] (28) at (-1.15, 2.875) {$c_{3}$};
		\node [style=none] (29) at (-1.225, 0.1) {$c_{4}$};
		\node [style=none] (30) at (2.25, 2.9) {$a_{1}$};
		\node [style=none] (31) at (2.45, 0.075) {$a_{4}$};
		\node [style=none] (32) at (5.825, 2.9) {$a_{2}$};
		\node [style=none] (33) at (-4.5, 3.5) {$b_{1}$};
		\node [style=none] (34) at (-4.5, -0.5) {$b_{2}$};
		\node [style=none] (35) at (4, -0.5) {$b_{2}$};
		\node [style=none] (36) at (4, 3.5) {$b_{1}$};
		\node [style=none] (37) at (5.75, 0.075) {$a_{3}$};
		\node [style=none] (38) at (-4.5, -1.25) {$(1)$};
		\node [style=none] (39) at (4, -1.25) {$(2)$};
		\node [style=none] (40) at (-6.3, 1.5) {$\alpha_{1}$};
		\node [style=none] (41) at (-2.65, 1.5) {$\alpha_{2}$};
		\node [style=none] (42) at (2.05, 1.5) {$\alpha_{1}$};
		\node [style=none] (43) at (6, 1.5) {$\alpha_{2}$};
	\end{pgfonlayer}
	\begin{pgfonlayer}{edgelayer}
		\draw [style=blackedge_thick] (0.center) to (4);
		\draw [style=blackedge_thick] (4) to (2.center);
		\draw [style=blackedge_thick] (5) to (3.center);
		\draw [style=blackedge_thick] (1.center) to (5);
		\draw [style=blackedge_thick] (0.center) to (6.center);
		\draw [style=blackedge_thick] (7.center) to (1.center);
		\draw [style=blackedge_thick] (2.center) to (9.center);
		\draw [style=blackedge_thick] (8.center) to (3.center);
		\draw [style=blackedge_thick, bend left=75, looseness=2.25] (13.center) to (20);
		\draw [style=blackedge_thick, bend left=75, looseness=2.25] (19) to (14.center);
		\draw [style=blackedge_thick] (11.center) to (19);
		\draw [style=blackedge_thick] (19) to (13.center);
		\draw [style=blackedge_thick, bend left=285, looseness=2.50] (19) to (12.center);
		\draw [style=blackedge_thick] (20) to (14.center);
		\draw [style=blackedge_thick] (12.center) to (20);
		\draw [style=blackedge_thick, bend right=75, looseness=2.25] (11.center) to (20);
	\end{pgfonlayer}

\end{tikzpicture}

  \caption{The face $F$ with $\ell_{b}=2$ and $\ell_{w}=4$}
  \label{caseK2,4}
\end{figure}
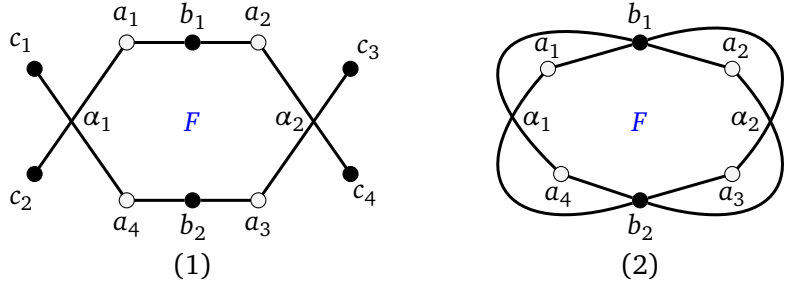

Lemma \ref{true-num} provides a foundation for obtaining more precise structural characterizations of the faces in $G$.
\begin{proposition}\label{seven}
Let $F$ be a face of $G$. Then $F$ belongs to one of the nine configurations shown in Figure~\ref{facein2-con}, from which it follows that $3 \le |F| \le 6$.
\end{proposition}
\begin{proof}
Recall that $G$ is $2$-connected, by Lemma \ref{2-connected property},
no true vertex on $\partial(F)$ is counted more than once.
Let $\ell_{t}$ denote the number of true vertices on $\partial (F)$.
By Lemma \ref{true-num}(ii), we have $2\leq \ell_{t} \leq 5$.
\setcounter{claim}{0}
\renewcommand{\thecase}{\arabic{claim}}
\begin{claim}
If $\ell_{t}=2$, $F$ is shown in Figure \ref{facein2-con} (1).
\end{claim}
\begin{proof}
If $\ell_{t}=2$, then $\ell_{f}\leq 1$ by Lemma \ref{true-num} (i).
If $\ell_{f}=0$, then $F$ is true 2-face, contradicting Lemma \ref{lem:facesbound}.
If $\ell_{f}=1$, then $F$ is a false $3$-face with two true vertices of different colors (i.e., white and black), 
as shown in Figure \ref{facein2-con} (1).
\end{proof}

\begin{claim}
If $\ell_{t}=3$, $F$ is shown in Figure \ref{facein2-con} (2) or (3).
\end{claim}

\begin{proof}
If $\ell_{t}=3$, by Lemma \ref{true-num} (i),  then $\ell_{f}\leq 1$.
Lemma \ref{lem:facesbound}  yields that $\ell_{f}\neq 0$.
Thus $\ell_{f}=1$.
By Lemma \ref{lem:consecutivevertices}, three consecutive vertices on $\partial(F)$ can not share the same color.
Therefore, only two cases need to be considered:
$(1)$ $\ell_{b}=1$ and $\ell_{w}=2$ or (2)  $\ell_{b}=2$ and $\ell_{w}=1$, 
where $\ell_{b}$ and $\ell_{w}$ are the numbers of black and white vertices on $\partial(F)$, respectively.
In both cases, $|\ell_{b}-\ell_{w}|=1$.
By Lemma \ref{properties of G} $(ii)$, 
$F$ is a false $4$-face where two end-vertices of the crossing on $\partial(F)$ share the same color, 
as shown in Figure \ref{facein2-con} $(2)$ or $(3)$.
\end{proof}


\begin{claim}
If $\ell_{t}=4$, $F$ is shown in Figure \ref{facein2-con} (4)-(7).
\end{claim}

If $\ell_{t}=4$, 
then $\ell_{b}=\ell_{w}=2$,
otherwise, it contradicts Lemma \ref{lem:consecutivevertices}.
By Lemma \ref{true-num} (i), we have $\ell_{f}\leq 2$.
If $\ell_{f}=0$, then $F$ is a true $4$-face, as shown in Figure \ref{facein2-con} (4).
If $\ell_{f}=1$, then $F$ is a false 5-face, as shown in Figure \ref{facein2-con} (5).
If $\ell_{f}=2$, let $c_{1}$ and $c_{2}$ be the two crossings on $\partial(F)$. 
If $c_{1}\neq c_{2}$, by IC-planarity, then $N_{G}(c_{1})\cap N_{G}(c_{2})=\varnothing$.
Thus $F$ is false 6-face in which both $c_{1}$ and $c_{2}$ have one end-vertex white and the other black on $\partial(F)$,
as shown in Figure \ref{facein2-con} $(6)$.
Otherwise, if the two end-vertices of $c_{1}$ on $\partial(F)$ have the same color, say white, 
then by IC-planarity, the two end-vertices of $c_{2}$ on $\partial(F)$ would be black.
In this case, $N_{G}(c_{1})\cap N_{G}(c_{2})\neq\varnothing$ by Lemma \ref{configuration}, a contradiction to IC-planarity.
If $c_{1}= c_{2}$,
then $F$ is a false 6-face, which is isomorphic to the $6$-patch of a tie, as shown in Figure \ref{facein2-con} $(7)$ or Figure \ref{tie} $(2)$.
\begin{claim}
If $\ell_{t}=5$,  $F$ is shown in Figure \ref{facein2-con} (8)-(9).
\end{claim}

If $\ell_{t}=5$,
by Lemma \ref{lem:consecutivevertices}, only two cases need to be considered:
$(1)$ $\ell_{b}=2$ and $\ell_{w}=3$ or $(2)$  $\ell_{b}=3$ and $\ell_{w}=2$.
By Lemma \ref{true-num} $(i)$, we obtain $\ell_{f}\leq 2$.
Note that in both cases, $|\ell_{b}-\ell_{w}|=1$,
Lemma \ref{properties of G} (ii) implies that $\ell_{f}\geq1$
and there exists one crossing $c_{0}$ on $\partial(F)$ with two end-vertices share the same color.
Consequently, by Lemma \ref{configuration}, we have $N_{G}(c_{0})\subseteq \partial(F)$.
Then there are four vertices on $\partial(F)$ adjacent to $c_{0}$.
Due to IC-planarity, there are no crossings other than $c_{0}$ on $\partial(F)$, thus $\ell_{f}=1$.
Then $F$ is a false 6-face with exactly one crossing $c_{0}$, where both end-vertices of $c_{0}$ on $\partial(F)$ share the same color,
as shown in Figure \ref{facein2-con} $(8)$ and $(9)$.

From the above analysis, if $F$ is a face of $G$, then $3\le |F|\le 6$.
\end{proof}

\begin{figure}[h!]
  \centering
\begin{tikzpicture}[scale=0.8]
	\begin{pgfonlayer}{nodelayer}
		\node [style=whitenode] (0) at (-8, 7) {};
		\node [style=blacknode] (1) at (-6, 7) {};
		\node [style=whitenode] (2) at (-7.75, 5) {};
		\node [style=blacknode] (3) at (-6.25, 5) {};
		\node [style=blacknode] (4) at (-3.5, 7) {};
		\node [style=whitenode] (5) at (-4.5, 6) {};
		\node [style=whitenode] (6) at (-2.5, 6) {};
		\node [style=blacknode] (7) at (-4, 5) {};
		\node [style=blacknode] (8) at (-3, 5) {};
		\node [style=whitenode] (9) at (2.5, 7) {};
		\node [style=blacknode] (10) at (2.5, 5) {};
		\node [style=blacknode] (11) at (4.5, 7) {};
		\node [style=whitenode] (12) at (4.5, 5) {};
		\node [style=blacknode] (13) at (6, 6.6) {};
		\node [style=whitenode] (14) at (8.075, 6.6) {};
		\node [style=whitenode] (15) at (6, 5.75) {};
		\node [style=blacknode] (16) at (8.075, 5.75) {};
		\node [style=whitenode] (17) at (6.425, 5) {};
		\node [style=blacknode] (18) at (7.7, 5) {};
		\node [style=whitenode] (19) at (-6.775, 3) {};
		\node [style=blacknode] (20) at (-5.375, 3) {};
		\node [style=blacknode] (21) at (-6.775, 1) {};
		\node [style=whitenode] (22) at (-5.375, 1) {};
		\node [style=whitenode] (23) at (-7.525, 2.525) {};
		\node [style=blacknode] (24) at (-7.525, 1.5) {};
		\node [style=blacknode] (25) at (-4.6, 2.525) {};
		\node [style=whitenode] (26) at (-4.6, 1.5) {};
		\node [style=whitenode] (27) at (-2.025, 3.05) {};
		\node [style=blacknode] (28) at (-0.525, 3.05) {};
		\node [style=whitenode] (29) at (-2.025, 1.55) {};
		\node [style=blacknode] (30) at (-0.525, 1.55) {};
		\node [style=whitenode] (31) at (2.975, 3) {};
		\node [style=blacknode] (32) at (1.975, 2.5) {};
		\node [style=blacknode] (33) at (3.975, 2.5) {};
		\node [style=whitenode] (34) at (2.225, 1.5) {};
		\node [style=whitenode] (35) at (3.725, 1.5) {};
		\node [style=none] (36) at (1.975, 2.5) {};
		\node [style=none] (37) at (3.975, 2.5) {};
		\node [style=none] (38) at (-7, 4) {$(1)$};
		\node [style=none] (39) at (-3.5, 4) {$(2)$};
		\node [style=none] (40) at (3.5, 4) {$(4)$};
		\node [style=none] (41) at (7, 4) {$(5)$};
		\node [style=none] (42) at (-6.025, 0) {$(6)$};
		\node [style=none] (43) at (-1.275, 0) {$(7)$};
		\node [style=none] (44) at (2.975, 0) {$(8)$};
		\node [style=none] (45) at (-7, 5.875) {};
		\node [style=none] (46) at (-3.5, 5.35) {};
		\node [style=none] (47) at (7.075, 5.3) {};
		\node [style=none] (48) at (-7.275, 2) {};
		\node [style=none] (49) at (-4.85, 2) {};
		\node [style=none] (50) at (-1.275, 2.3) {};
		\node [style=none] (51) at (2.975, 1.175) {};
		\node [style=none] (52) at (-7, 6.55) {\color{blue}$F$};
		\node [style=none] (53) at (-3.5, 6.175) {\color{blue}$F$};
		\node [style=none] (54) at (3.5, 6) {\color{blue}$F$};
		\node [style=none] (55) at (7, 6.025) {\color{blue}$F$};
		\node [style=none] (56) at (-6.025, 2) {\color{blue}$F$};
		\node [style=none] (57) at (-2.025, 2.25) {\color{blue}$F$};
		\node [style=none] (58) at (2.975, 2) {\color{blue}$F$};
		\node [style=whitenode] (59) at (0.075, 7) {};
		\node [style=blacknode] (60) at (-0.925, 6) {};
		\node [style=blacknode] (61) at (1.075, 6) {};
		\node [style=whitenode] (62) at (-0.425, 5) {};
		\node [style=whitenode] (63) at (0.575, 5) {};
		\node [style=none] (64) at (0.075, 4) {$(3)$};
		\node [style=none] (65) at (0.075, 5.35) {};
		\node [style=none] (66) at (0.075, 6.175) {\color{blue}$F$};
		\node [style=blacknode] (67) at (7.05, 3) {};
		\node [style=whitenode] (68) at (6.05, 2.5) {};
		\node [style=whitenode] (69) at (8.05, 2.5) {};
		\node [style=blacknode] (70) at (6.3, 1.5) {};
		\node [style=blacknode] (71) at (7.8, 1.5) {};
		\node [style=none] (72) at (6.05, 2.5) {};
		\node [style=none] (73) at (8.05, 2.5) {};
		\node [style=none] (74) at (7.05, 0) {$(9)$};
		\node [style=none] (75) at (7.05, 1.175) {};
		\node [style=none] (76) at (7.05, 2) {\color{blue}$F$};
		\node [style=none] (77) at (2.975, 0.8) {};
		\node [style=none] (78) at (7.05, 0.8) {};
	\end{pgfonlayer}
	\begin{pgfonlayer}{edgelayer}
\fill[blue!10] (0.center) to (1.center) to (45.center) to (0.center);
\fill[blue!10, even odd rule] (4.center) to (6.center) to (46.center) to (5.center) to (4.center);
\fill[blue!10] (59.center) to (61.center) to (65.center) to (60.center) to (59.center);
\fill[blue!10] (9.center) to (11.center) to (12.center) to (10.center) to (9.center);
\fill[blue!10] (13.center) to (14.center) to (16.center) to (47.center) to (15.center) to (13.center);
\fill[blue!10] (19.center) to (20.center) to (49.center) to (22.center) to (21.center) to (48.center) to (19.center);
\fill [fill=blue!10] (27.center) to (50.center) to (29.center);
\fill [fill=blue!10] (28.center) to (50.center) to (30.center);
\fill[blue!10] (28.center) to (30.center) to (29.center) to (27.center) [bend right=135, looseness=7.25] to (28.center);
\fill[blue!10] (31.center) to (37.center) to (35.center) [bend left=15, looseness=1.25] to (77.center) [bend left=15, looseness=1.25] to (34.center); 
\fill[blue!10] (36.center) to  (31.center) to (34.center) to (36.center);
\fill[blue!10] (67.center) to (73.center) to (71.center) [bend left=15, looseness=1.25] to (78.center) [bend left=15, looseness=1.25] to (70.center); 
\fill[blue!10] (72.center) to  (67.center) to (70.center) to (72.center);
		\draw [style={blackedge_thick}] (0) to (1);
		\draw [style={blackedge_thick}] (1) to (2);
		\draw [style={blackedge_thick}] (0) to (3);
		\draw [style={blackedge_thick}] (5) to (4);
		\draw [style={blackedge_thick}] (4) to (6);
		\draw [style={blackedge_thick}] (5) to (8);
		\draw [style={blackedge_thick}] (6) to (7);
		\draw [style={blackedge_thick}] (9) to (11);
		\draw [style={blackedge_thick}] (11) to (12);
		\draw [style={blackedge_thick}] (9) to (10);
		\draw [style={blackedge_thick}] (10) to (12);
		\draw [style={blackedge_thick}] (13) to (14);
		\draw [style={blackedge_thick}] (13) to (15);
		\draw [style={blackedge_thick}] (14) to (16);
		\draw [style={blackedge_thick}] (15) to (18);
		\draw [style={blackedge_thick}] (16) to (17);
		\draw [style={blackedge_thick}] (19) to (20);
		\draw [style={blackedge_thick}] (21) to (22);
		\draw [style={blackedge_thick}] (19) to (24);
		\draw [style={blackedge_thick}] (23) to (21);
		\draw [style={blackedge_thick}] (20) to (26);
		\draw [style={blackedge_thick}] (25) to (22);
		\draw [style={blackedge_thick}, bend right=135, looseness=6.25] (27) to (28);
		\draw [style={blackedge_thick}] (28) to (29);
		\draw [style={blackedge_thick}] (29) to (30);
		\draw [style={blackedge_thick}] (27) to (30);
		\draw [style={blackedge_thick}] (32) to (31);
		\draw [style={blackedge_thick}] (31) to (33);
		\draw [style={blackedge_thick}] (32) to (34);
		\draw [style={blackedge_thick}] (33) to (35);
		\draw [style={blackedge_thick}] (60) to (59);
		\draw [style={blackedge_thick}] (59) to (61);
		\draw [style={blackedge_thick}] (60) to (63);
		\draw [style={blackedge_thick}] (61) to (62);
		\draw [style={blackedge_thick}] (68) to (67);
		\draw [style={blackedge_thick}] (67) to (69);
		\draw [style={blackedge_thick}] (68) to (70);
		\draw [style={blackedge_thick}] (69) to (71);
		\draw [style={blackedge_thick}, bend right=90, looseness=2.25] (34) to (37.center);
		\draw [style={blackedge_thick}, bend right=270, looseness=2.25] (35) to (36.center);
		\draw [style={blackedge_thick}, bend right=90, looseness=2.25] (70) to (73.center);
		\draw [style={blackedge_thick}, bend right=270, looseness=2.25] (71) to (72.center);
		\draw [style={blackedge_thick}, bend left=135, looseness=2.75] (15) to (16);
		\draw [style={blackedge_thick}, bend right=105, looseness=2.00] (19) to (21);
		\draw [style={blackedge_thick}, bend left=105, looseness=2.00] (20) to (22);
	\end{pgfonlayer}
\end{tikzpicture}
 \caption{All possible faces $F$ (colored in purple in the electronic version) in $G$}
  \label{facein2-con}
\end{figure}
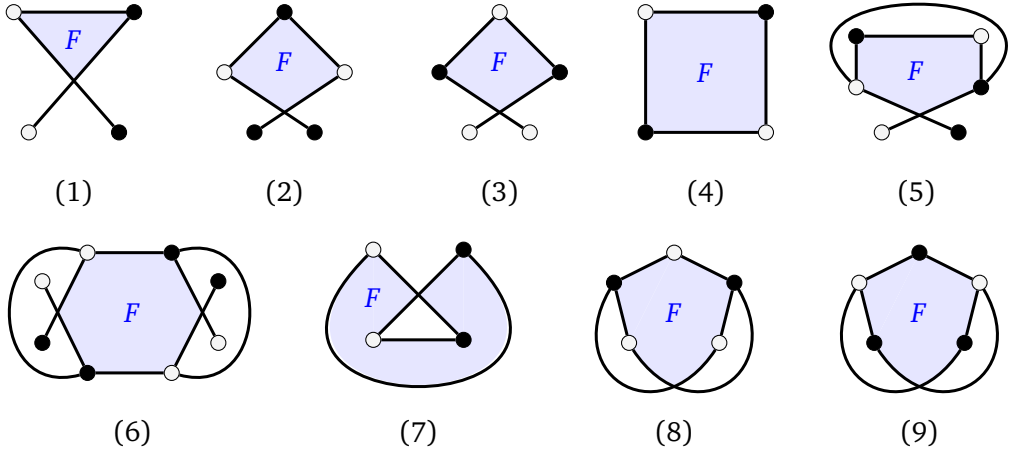

Based on Proposition~\ref{seven}, we can classify the set of faces in $G$ according to their sizes and structural characteristics.
Throughout the rest of this paper, we denote by $\mathcal{F}_{k}$ the set of $k$-faces in $G$.

For 4-faces of $G$, we define:
\begin{itemize}
    \item $\mathcal{A}_{4}$: the set of 4-faces shown in Figure~\ref{facein2-con} (2),
    \item $\mathcal{B}_{4}$: the set of 4-faces shown in Figure~\ref{facein2-con} (3),
    \item $\mathcal{D}_{4}$: the set of 4-faces shown in Figure~\ref{facein2-con} (4).
\end{itemize}

Therefore, these satisfy the following:
\begin{itemize}
    \item Each face in $\mathcal{A}_{4}$ is a false 4-face whose boundary contains a black vertex, two white vertices, and a crossing.
    \item Each face in $\mathcal{B}_{4}$ is a false 4-face whose boundary contains a white vertex, two black vertices, and a crossing.
    \item Each face in $\mathcal{D}_{4}$ is a true 4-face whose boundary contains four clean edges.
\end{itemize}

For 6-faces of $G$, we define:
\begin{itemize}
    \item $\mathcal{A}_{6}$: the set of 6-faces shown in Figure~\ref{facein2-con} (6),
    \item $\mathcal{B}_{6}$: the set of 6-faces shown in Figure~\ref{facein2-con} (7),
    \item $\mathcal{D}_{6}$: the set of 6-faces shown in Figure~\ref{facein2-con} (8),
    \item $\mathcal{E}_{6}$: the set of 6-faces shown in Figure~\ref{facein2-con} (9).
\end{itemize}

These satisfy the following:
\begin{itemize}
    \item Each face in $\mathcal{A}_{6}$ is a false 6-face whose boundary contains two white vertices, two black vertices, and two crossings.
    \item Each face in $\mathcal{B}_{6}$ is a false 6-face whose boundary contains two white vertices, two black vertices, and one crossing (counted twice).
    \item Each face in $\mathcal{D}_{6}$ is a false 6-face whose boundary contains three white vertices, two black vertices, and one crossing.
    \item Each face in $\mathcal{E}_{6}$ is a false 6-face whose boundary contains three black vertices, two white vertices, and one crossing.
\end{itemize}

The following corollary follows from Proposition~\ref{seven}.
\begin{corollary}
$|\mathcal{F}_{4}| = |\mathcal{A}_{4}| + |\mathcal{B}_{4}| + |\mathcal{D}_{4}|$ and $
|\mathcal{F}_{6}| = |\mathcal{A}_{6}| + |\mathcal{B}_{6}| + |\mathcal{D}_{6}| + |\mathcal{E}_{6}|.
$
\end{corollary}

\subsection{Properties of two special  types of ties}

\noindent Let $T(\alpha)$ be a tie in $G$. Then we define two special ties in $G$:
\begin{itemize}
    \item $T(\alpha)$ is \emph{clean} if both $3$-patches are faces of $G$;
    \item $T(\alpha)$ is \emph{bad} if the $6$-patch and exactly one of the $3$-patches are faces of $G$.
\end{itemize}

For example, in Figure \ref{tie3}, $T(\alpha_{2})$ is a clean tie and $T(\alpha_{3})$ is a bad tie in $G$.

A face $F$ is {\it incident} with a tie $T(\alpha)$  if $\alpha \in \partial(F)$.
Based on Proposition~\ref{seven}, we derive the following relation between the bad and clean ties \(T(\alpha)\) and face types in \(G\).
\begin{proposition}
\label{tie and face}
In $G$, the following statements hold.
\begin{itemize}

\item[(i)] Each bad tie of $G$ is incident with exactly one false 3-face in $\mathcal{\mathcal{F}}_{3}$, exactly one false 6-face in $\mathcal{B}_{6}$,
and exactly one false face in $\mathcal{\mathcal{F}}_{5}\cup \mathcal{\mathcal{A}}_{6}$;

\item[(ii)] each clean tie of  $G$ is incident with exactly two false 3-faces in $\mathcal{F}_{3}$ and
at most one false 6-face in $\mathcal{D}_{6}\cup \mathcal{E}_{6}$;

\item[(iii)] each face in $\mathcal{F}_{5}\cup \mathcal{B}_{6}$ is incident with exactly one bad tie, and each face in  $\mathcal{A}_{6}$ is incident with exactly two bad ties; and

\item[(iv)] each face in $\mathcal{A}_{4}\cup \mathcal{B}_{4}\cup \mathcal{D}_{6}\cup \mathcal{E}_{6}$ is incident with exactly one clean tie.
\end{itemize}
\end{proposition}

\begin{proof}
Prove (i) and (ii).

Let $ab$ and $cd$ be two edges in $G$ that cross at a point $\alpha$, and let $T(\alpha)$ be a tie of $G$, where $\{a,c\}\subseteq X$ and $\{b,d\}\subseteq Y$.
Assume that $T(\alpha)$ is bad in $G$.
By the definition of a bad tie, 
$T(\alpha)$ is incident to a false $3$-face (say $ad\alpha a$) in $\mathcal{\mathcal{F}}_{3}$ and a false $6$-face (i.e., $ad\alpha bc\alpha a$) in $\mathcal{B}_{6}$, 
as shown in Figure \ref{facein2-con} $(1)$ and $(7)$.
Since the remaining $3$-patch of $T(\alpha)$ is not a face in $G$, the two edges segments $b\alpha$ and $c\alpha$ lie on the boundary of a face $F$ with $|F|\geq4$.
As vertices $b$ and $c$ share different colors, 
by Lemma \ref{seven}, it is not difficult to verify that the other face incident to $\alpha$ can only be a face in $\mathcal{F}_{5}$ or $\mathcal{A}_{6}$ , 
as shown in Figure \ref{facein2-con} $(5)$ and $(6)$.
Then (i) holds.

Assume that $T(\alpha)$ is clean in $G$.
By the definition of a clean tie, $T(\alpha)$ is incident to two false 3-faces (i.e., $ad\alpha a$ and $bc\alpha b$) in $\mathcal{\mathcal{F}}_{3}$.
Since the 6-patch of $T(\alpha)$ is not a face in $G$, there exists at least one vertex within it.
If $F$ is a face $F$ such that $\partial(F)$ contains the boundary walk $bc\alpha ad$ or $ad\alpha bc$.
By Lemma \ref{seven}, $F$ can only be a face in $\mathcal{D}_{6}$ or $\mathcal{E}_{6}$, 
as shown in  Figure \ref{facein2-con} $(8)$ and $(9)$. Then (ii) holds.

Prove (iii) and (iv).

Let $F$ be a face in $\mathcal{F}_{5}$, as shown in Figure \ref{facein2-con} (5).
By $(i)$ and the fact that $\partial(F)$ contains exactly one crossing $\alpha$, $F$ is incident with exactly one bad tie.
Similarly, assume $F$ is a face in $\mathcal{B}_{6}$, as shown in Figure \ref{facein2-con} $(7)$.
Since $\partial(F)$ contains exactly one crossing (counted twice), 
combing with (i), $F$ is incident with exactly one bad tie.
If $F$ is a face in $\mathcal{A}_{6}$, as shown in Figure \ref{facein2-con} $(6)$.
By (i) and the fact that $\partial(F)$ contains exactly two crossings, each with end-vertices of distinct colors,
$F$ is incident with exactly two bad ties. Thus (iii) holds.

Let $F$ be a face in $\mathcal{A}_{4}\cup \mathcal{B}_{4}\cup \mathcal{D}_{6}\cup \mathcal{E}_{6}$.
Note $\partial(F)$ contains exactly one crossing $\alpha$ with end-vertices share the same color, 
then $F$ is incident with exactly one clean tie.
Thus (iv) holds.
\end{proof}

\begin{figure}
\centering
\begin{tikzpicture}[scale=0.75]
	\begin{pgfonlayer}{nodelayer}
		\node [style=none] (11) at (11, 1.25) {};
		\node [style=blacknode] (22) at (5.025, 3.5) {};
		\node [style=whitenode] (23) at (9.875, 3.5) {};
		\node [style=whitenode] (24) at (9.875, -1) {};
		\node [style=blacknode] (25) at (5.025, -1) {};

		\node [style=none] (31) at (6.8, 1.25) {$\alpha_{1}$};
		\node [style=whitenode] (32) at (7.025, 0.325) {};
		\node [style=blacknode] (33) at (7.875, 0.325) {};
		\node [style=whitenode] (34) at (7.025, -0.3) {};
		\node [style=blacknode] (35) at (7.875, -0.3) {};
		\node [style=whitenode] (36) at (7, 3) {};
		\node [style=blacknode] (37) at (7.9, 3) {};
		\node [style=whitenode] (38) at (7, 2.125) {};
		\node [style=blacknode] (39) at (7.9, 2.125) {};
		\node [style=none] (40) at (7.95, 2.55) {$\alpha_{2}$};
		\node [style=none] (41) at (7.425, 0.475) {$\alpha_{3}$};
	\end{pgfonlayer}
	\begin{pgfonlayer}{edgelayer}
		\draw [style={blackedge_thick}] (23) to (25);
		\draw [style={blackedge_thick}] (25) to (24);
		\draw [style={blackedge_thick}] (22) to (24);
		\draw [style={blackedge_thick}] (22) to (23);
		\draw [style={blackedge_thick}] (32) to (35);
		\draw [style={blackedge_thick}] (33) to (34);
		\draw [style={blackedge_thick}] (34) to (35);
		\draw [style={blackedge_thick}] (32) to (25);
		\draw [style={blackedge_thick}] (33) to (24);
		\draw [style={blackedge_thick}, in=315, out=-135, looseness=4.75] (32) to (33);
		\draw [style={blackedge_thick}] (36) to (37);
		\draw [style={blackedge_thick}] (37) to (38);
		\draw [style={blackedge_thick}] (38) to (39);
		\draw [style={blackedge_thick}] (36) to (39);
		\draw [style={blackedge_thick}] (36) to (22);
		\draw [style={blackedge_thick}] (38) to (22);
		\draw [style={blackedge_thick}] (37) to (23);
		\draw [style={blackedge_thick}] (39) to (23);
	\end{pgfonlayer}
\end{tikzpicture}
  \caption{ A graph $G$ with three ties}
  \label{tie3}
\end{figure}
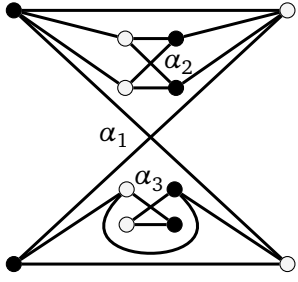

\subsection{Structural properties on ties, faces, and $G_{p}$ of $G$ with $\kappa(G) \ge 3$ }

\noindent As the connectivity of $G$ increases, the following lemma shows that stronger constraints are enforced  on structures (such as  ties and faces) of $G$.
\begin{proposition}\label{3-connected property}
If $\kappa(G)\ge 3$,  then the following statements hold.

\begin{itemize}
  \item [(i)] each tie in $G$ is clean;
  \item [(ii)] any  face in $G$ only has four possibilities, as shown in Figure \ref{facein2-con} $(1)$-$(4)$;
  \item [(iii)] each pair of crossing edges is contained in a region bounded by a clean $6$-cycle in $G$; and
  \item [(iv)]$G_{p}$ is a quadrangulation.
\end{itemize}

\end{proposition}
\begin{proof}
(i). Assume in $G$, edges $ab$ and $cd$ cross at point $\alpha$, where $\{a,c\}\in X$ and $\{b,d\}\in Y$.
By Lemma \ref{cleantie}, $\{ad,bc\}\subseteq E(G)$ and they are clean in $G$.
If there exists a vertex within the $3$-patch bounded by $a\alpha da$,
then $\{a,d\}$ forms a $2$-cut of $G$,
which contradicts that $G$ is $3$-connected.
Hence, each tie in $G$ is clean.

(ii). By Proposition  \ref{seven},  there are nine possible faces in $G$.
Since $G$ is $3$-connected, according to (i), each pair of crossing edges in $G$ forms a clean tie.
Thus there are no vertices within any $3$-patch of each tie.
Consequently, it is not difficult to verify that the faces depicted in Figure \ref{facein2-con} $(5)$, $(6)$ and $(7)$ can not exist in $G$.
Furthermore, the two black vertices in Figure \ref{facein2-con} $(8)$ form a $2$-cut of $G$, contradicts $G$ is $3$-connected.
Similarly, the two white vertices in Figure \ref{facein2-con} $(9)$ form a $2$-cut of $G$.
Therefore, if $G$ is $3$-connected, the faces in $G$ can only be those illustrated in Figure \ref{facein2-con} $(1)$-$(4)$.

(iii). By (i), each pair of crossing edges in $G$ forms a clean tie.
Then each clean tie in $G$ is incident to precisely two false $3$-faces by Proposition \ref{tie and face} $(ii)$.
According to (ii), the faces in $G$ can only be those shown in Figure \ref{facein2-con} (1)-(4).
As $G$ is $3$-connected, we have $n(G)\geq6$.
Thus, the remaining two false faces incident to each tie must be one face in $\mathcal{A}_{4}$ and one face in $\mathcal{B}_{4}$.
Hence, each pair of crossing edges is contained in a $6$-cycle and it is clean by IC-planarity.

(iv). By (ii), every true face in $G$ is a $4$-face as shown in Figure \ref{facein2-con} (4). 
Furthermore, (iii) asserts that each pair of crossing edges is contained in a clean $6$-cycle of $G$.
Thus, by removing one edge arbitrarily from each crossing pair,
the resulting graph $G_{p}$ is a plane graph in which each face is a 4-face, as desired.
\end{proof}

\section{Characterization of MBICP-graphs without crossings}\label{crossing number0}
\noindent In this section, we characterize the MBICP-graphs without crossings.
\begin{theorem}\label{cr=0}
If $G$ is an $n$-vertex MBICP-graph without crossings. 
Then $G\cong K_{1,n-1}$ or $G\cong K_{2,n-2}$, as shown in Figure \ref{K2,n-2}.
\end{theorem}
\begin{proof}
Let $s$ and $t$ denote the sizes of the bipartition sets of $G$, where $s\leq t$.
For $s=1$, the star graph $K_{1,n-1}$ trivially constitutes a class of MBICP-graphs without crossings,
as its good drawing is unique, in which no edges crossed each other, as illustrated in Figure \ref{K2,n-2} $(1)$.

For $s=2$, the complete bipartite graph $K_{2,n-2}$ admits a planar drawing as illustrated in Figure \ref{K2,n-2} $(2)$.
Observe that no edges from the complement of $G$ can be added without violating the assumption of $G$.
Thus, the graph $K_{2,n-2}$ shown in Figure \ref{K2,n-2} $(2)$ is a MBICP- graph without crossings.

Notably, any two adjacent $4$-faces in Figure \ref{K2,n-2} $(2)$ share exactly two common edges.

Now we consider $s\geq3$. Clearly, $n\geq6$. 
By the assumption of $G$, $G$ is a quadrangulation.
We are now going to prove the following claim.
\begin{figure}[h!]
  \centering
\begin{tikzpicture}[scale=0.6]
	\begin{pgfonlayer}{nodelayer}
		\node [style=whitenode] (0) at (0, 2) {};
		\node [style=whitenode] (1) at (0, -2) {};
		\node [style=blacknode] (2) at (-4, 0) {};
		\node [style=blacknode] (3) at (-2.5, 0) {};
		\node [style=blacknode] (4) at (-1, 0) {};
		\node [style=blacknode] (5) at (0.5, 0) {};
		\node [style=blacknode] (6) at (2, 0) {};
		\node [style=blacknode] (7) at (3.5, 0) {};
		\node [style=blacknode] (8) at (6, 0) {};
		\node [style=blacknode_v1] (9) at (4.25, 0) {};
		\node [style=blacknode_v1] (10) at (4.75, 0) {};
		\node [style=blacknode_v1] (11) at (5.25, 0) {};
        \node [style=whitenode] (12) at (-9, 0) {};
		\node [style=blacknode] (13) at (-9, 2) {};
		\node [style=blacknode] (14) at (-7.25, 1) {};
		\node [style=blacknode] (15) at (-9, -2) {};
		\node [style=blacknode] (16) at (-10.75, -1) {};
		\node [style=blacknode] (17) at (-7.25, -1.25) {};
		\node [style=blacknode] (18) at (-10.75, 1) {};
		\node [style=blacknode_v1] (19) at (-10.775, 0.45) {};
		\node [style=blacknode_v1] (20) at (-10.85, 0) {};
		\node [style=blacknode_v1] (21) at (-10.8, -0.375) {};
		\node [style=none] (22) at (-9, -3) {$(1)$ $K_{1,n-1}$};
		\node [style=none] (23) at (0, -3) {$(2)$ $K_{2,n-2}$};
	\end{pgfonlayer}
	\begin{pgfonlayer}{edgelayer}
		\draw [style=blackedge_thick] (0.center) to (2);
		\draw [style=blackedge_thick] (0.center) to (3);
		\draw [style=blackedge_thick] (0.center) to (4);
		\draw [style=blackedge_thick] (0.center) to (5);
		\draw [style=blackedge_thick] (0.center) to (6);
		\draw [style=blackedge_thick] (0.center) to (7);
		\draw [style=blackedge_thick] (0.center) to (8);
		\draw [style=blackedge_thick] (2) to (1.center);
		\draw [style=blackedge_thick] (3) to (1.center);
		\draw [style=blackedge_thick] (4) to (1.center);
		\draw [style=blackedge_thick] (5) to (1.center);
		\draw [style=blackedge_thick] (6) to (1.center);
		\draw [style=blackedge_thick] (7) to (1.center);
		\draw [style=blackedge_thick] (8) to (1.center);
		\draw [style=blackedge_thick] (12.center) to (13);
		\draw [style=blackedge_thick] (12.center) to (14);
		\draw [style=blackedge_thick] (12.center) to (15);
		\draw [style=blackedge_thick] (12.center) to (16);
		\draw [style=blackedge_thick] (12.center) to (17);
		\draw [style=blackedge_thick] (18) to (12.center);
	\end{pgfonlayer}
\end{tikzpicture}
  \caption{MBICP-graphs without crossings}
  \label{K2,n-2}
\end{figure}
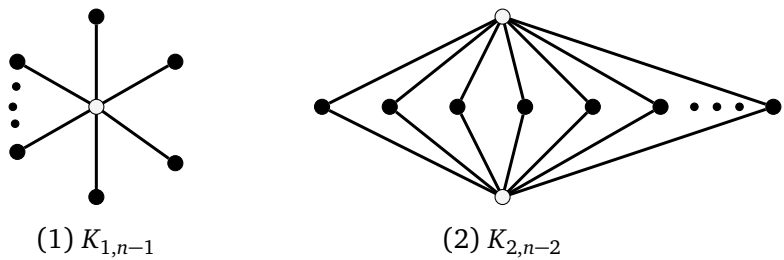
\setcounter{claim}{0}
\renewcommand{\thecase}{\arabic{claim}}
\begin{claim}\label{one common}
There exist two 4-faces in $G$ whose boundaries share exactly one common edge.
\end{claim}
\begin{proof}
It is clear that no two  4-faces in $G$ whose boundaries share four common edges, 
since this would imply $n=4$ and $G\cong K_{2,2}$, a contradiction.
If $F_{1}$ and $F_{2}$ are two 4-faces in $G$ whose boundaries share three common edges,
then $G$ would have parallel edges, contradicting the simplicity of $G$.
Thus, any two adjacent $4$-faces in $G$ share at most two common edges.
Suppose that any two adjacent 4-faces share exactly two common edges.
Then these two common edges must be adjacent in $G$,
otherwise their boundaries would have four vertices in common, which is impossible.
Thus, in this case, $G\cong K_{2,n-2}$, as shown in Figure \ref{K2,n-2} $(2)$, a contradiction to $s\geq3$.
Therefore, the claim holds.
\end{proof}
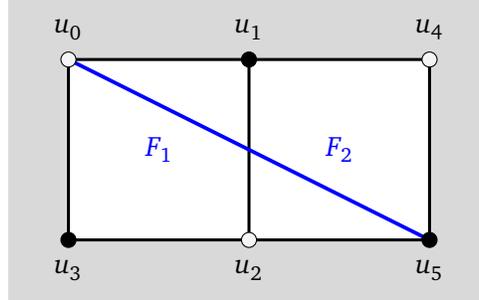
\begin{figure}[h!]
  \centering
  \begin{tikzpicture}[scale=0.8, pattern1/.style={ pattern color=black!60, pattern=north east lines}]
	\begin{pgfonlayer}{nodelayer}
		\node [style=whitenode] (0) at (-3, 3) {};
		\node [style=blacknode] (1) at (-3, 0) {};
		\node [style=blacknode] (2) at (0, 3) {};
		\node [style=whitenode] (3) at (0, 0) {};
		\node [style=whitenode] (4) at (3, 3) {};
		\node [style=blacknode] (5) at (3, 0) {};
		\node [style=none] (6) at (-4, 4) {};
		\node [style=none] (7) at (-4, -1) {};
		\node [style=none] (8) at (4, -1) {};
		\node [style=none] (9) at (4, 4) {};
		\node [style=none] (10) at (-1.5, 1.5) {\color{blue}$F_{1}$};
		\node [style=none] (11) at (1.5, 1.5) {\color{blue}$F_{2}$};

		\node [style=none] (12) at (-3, 3.5) {$u_{0}$};
		\node [style=none] (13) at (0, 3.5) {$u_{1}$};
		\node [style=none] (14) at (3, 3.5) {$u_{4}$};
		\node [style=none] (15) at (-3, -0.5) {$u_{3}$};
		\node [style=none] (16) at (0, -0.5) {$u_{2}$};
		\node [style=none] (17) at (3, -0.5) {$u_{5}$};
	\end{pgfonlayer}
	\begin{pgfonlayer}{edgelayer}
\fill[gray!30,even odd rule] (6.center) to (9.center) to (8.center) to (7.center) to (6.center) (0.center) to (2.center) to (4.center) to (5.center)to (3.center) to (1.center);
		\draw [style=blackedge_thick] (0.center) to (2.center);
		\draw [style=blackedge_thick] (2.center) to (4.center);
		\draw [style=blackedge_thick] (4.center) to (5.center);
		\draw [style=blackedge_thick] (2.center) to (3.center);
		\draw [style=blackedge_thick] (3.center) to (5.center);
		\draw [style=blackedge_thick] (0.center) to (1.center);
		\draw [style=blackedge_thick] (1.center) to (3.center);
		\draw [style=blue_thick] (0.center) to (5.center);

	\end{pgfonlayer}
\end{tikzpicture}
  \caption{Two adjacent 4-faces $F_{1}$ and $F_{2}$ in $G$ whose boundaries share exactly one common edge }
  \label{twofouradjaentface}
\end{figure}

By Claim \ref{one common}, let $F_{1}=u_{0}u_{1}u_{2}u_{3}u_{0}$ and $F_{2}=u_{1}u_{4}u_{5}u_{2}u_{1}$ 
be two adjacent 4-faces in $G$ whose boundaries share exactly one common edge $u_{1}u_{2}$,
where $\{u_{1},u_{3},u_{5}\}\in X$  and $\{u_{0},u_{2},u_{4}\}\in Y$.
Since $G$ has no crossings, edges $u_{0}u_{5}$ and $u_{3}u_{4}$ can not exist in $G$ simultaneously.
Without loss of generality, assume that $u_{3}u_{4}\in E(G)$ and $u_{0}u_{5}\notin E(G)$.
Then a new edge $u_{0}u_{5}$ (colored in blue) can be added within the two faces $F_{1}$ and $F_{2}$, as shown in Figure \ref{twofouradjaentface}.
In this case, $u_{0}u_{5}$ crosses $u_{1}u_{2}$.
Note that the resulting graph is still bipartite IC-plane, which
contradicts the maximality of $G$.
Then the theorem holds.
\end{proof}

Combining Theorem~\ref{cr=0} with the facts that $\kappa(K_{1,n-1}) = 1$ and $\kappa(K_{2,n-2}) = 2$ where $n\ge 2$, the following result follows immediately.

\begin{corollary}\label{3-con cr=0}
Let $G$ be a 3-connected MBICP-graph. Then $G$ has at least one crossing.
\end{corollary}

\section{ Proofs of Theorems \ref{2-connected} and \ref{3-connected}}\label{proofmaintheorem}

\noindent The whole section contributes to the proofs of Theorem \ref{2-connected} and Theorem \ref{3-connected}.

\subsection{Proof of Theorem \ref{2-connected}}
\begin{proof}

We will prove this theorem by induction on $n$.
Since $\kappa(G)\ge 2$, it follows that $n\geq4$.
If $n=4$, then $G\cong K_{2,2}$. In this case, $e(G)=4=\frac{3}{2}\times 4-2$. Thus the theorem holds for $n=4$.
Similarly, if $n=5$, then $G\cong K_{2,3}$, and  $e(G)=6 > \frac{3}{2}\times 5-2$. The theorem holds for $n=5$.

\setcounter{case}{0}
\renewcommand{\thecase}{\arabic{case}}
Now assume that the theorem holds for all $2$-connected MBICP-graphs of order less than $n$, where $n\geq6$.
Let $G$ be a $2$-connected MBICP-graph of order $n$.
Let $D$ be the corresponding maximal bipartite IC-planar drawing of $G$.

If $cr(D)=0$,  by Theorem \ref{cr=0}, we have $G\cong K_{2,n-2}$.
Then $e(G)=2n-4 \geq \frac{3}{2}n-2$ for $n\geq6$.
The theorem holds in this case. Thus we now proceed to consider $cr(D)\geq1$.

\begin{case}\label{case1}
For each crossing $\alpha$, the tie $T(\alpha)$ is clean  in $D$.
\end{case}
\begin{proof}
By Proposition \ref{seven}, there are nine types of possible faces in $D$, as depicted in Figure \ref{facein2-con}.
By the assumption, for each crossing $\alpha$, $T(\alpha)$ is  clean  in $D$.
This implies there are no vertices within any $3$-patch of each tie in $D$. Consequently, the faces of $D$ fall into six possible types, illustrated in Figure \ref{facein2-con} (1)–(4), (8) and (9),  where the 6-faces can only be of types Figure \ref{facein2-con} (8) or (9).

Assume that $D$ has no $6$-faces in $\mathcal{D}_{6}\cup \mathcal{E}_{6}$ as shown in Figure \ref{facein2-con} $(8)$ and $(9)$.
Then in $D$, each true face is a $4$-face and each pair of crossing edges is contained in a clean $6$-cycle. 
By deleting one edge from each pair of crossing edges in \( D \), we obtain a new graph \( D_p \).
Then $D_{p}$ is a quadrangulation. 
Thus, $e(D_{p})=2n(D_{p})=2n-4$. Furthermore, $cr(D)\geq1$, we have $e(G)$=$e(D_{p})+cr(D)\geq 2n-3 \geq \frac{3}{2}n-2$ for $n\geq2$.
Hence, the result holds.
\begin{figure}[h!]
  \centering
\begin{tikzpicture}[scale=0.6]
	\begin{pgfonlayer}{nodelayer}
		\node [style=whitenode] (0) at (-3.75, 6.75) {};
		\node [style=blacknode] (1) at (-5.25, 5.5) {};
		\node [style=blacknode] (2) at (-2.25, 5.5) {};
		\node [style=whitenode] (3) at (-4.75, 4) {};
		\node [style=whitenode] (4) at (-2.75, 4) {};
		\node [style=whitenode] (5) at (3.5, 7.25) {};
		\node [style=blacknode] (6) at (2, 6) {};
		\node [style=blacknode] (7) at (5, 6) {};
		\node [style=whitenode] (8) at (2.5, 4.5) {};
		\node [style=whitenode] (9) at (4.5, 4.5) {};
		\node [style=whitenode] (10) at (3.5, 2.5) {};
		\node [style=none] (11) at (-5.5, 5.75) {$c$};
		\node [style=none] (12) at (-1.8, 5.75) {$d$};
		\node [style=none] (13) at (-5, 3.75) {$a$};
		\node [style=none] (14) at (-2.5, 3.75) {$b$};
		\node [style=none] (15) at (-3.75, 3.95) {$\alpha$};
		\node [style=none] (16) at (-3.75, 7.25) {$x$};
		\node [style=none] (17) at (1.5, 6.25) {$c$};
		\node [style=none] (18) at (5.5, 6.25) {$d$};
		\node [style=none] (19) at (2.25, 4.25) {$a$};
		\node [style=none] (20) at (4.75, 4.25) {$b$};
		\node [style=none] (21) at (3.5, 4.35) {$\alpha$};
		\node [style=none] (22) at (3.5, 7.75) {$x$};
		\node [style=none] (23) at (3.5, 2) {$y$};
		\node [style=none] (24) at (-3.75, 5.25) {\color{blue}$F$};
		\node [style=none] (25) at (3.5, 5.75) {\color{blue}$F$};
		\node [style=none] (26) at (-3.75, 1) {$(1)$};
		\node [style=none] (27) at (3.5, 1) {$(2)$};
		\node [style=none] (28) at (3.5, 3.25) {\color{blue}$F'$};
	\end{pgfonlayer}
	\begin{pgfonlayer}{edgelayer}
		\draw [style={blackedge_thick}] (0) to (1);
		\draw [style={blackedge_thick}] (0) to (2);
		\draw [style={blackedge_thick}] (1) to (3);
		\draw [style={blackedge_thick}] (2) to (4);
		\draw [style={blackedge_thick}, in=-150, out=-105, looseness=2.50] (1) to (4);
		\draw [style={blackedge_thick}, in=285, out=-30, looseness=2.50] (3) to (2);
		\draw [style={black_bold}] (5) to (6);
		\draw [style={black_bold}] (5) to (7);
		\draw [style={blackedge_thick}] (6) to (8);
		\draw [style={blackedge_thick}] (7) to (9);
		\draw [style={blackedge_thick}, in=-150, out=-105, looseness=2.50] (6) to (9);
		\draw [style={blackedge_thick}, in=285, out=-30, looseness=2.50] (8) to (7);
		\draw [style={black_bold}, bend right=60, looseness=1.75] (6) to (10);
		\draw [style={black_bold}, bend left=60, looseness=1.75] (7) to (10);
	\end{pgfonlayer}
\end{tikzpicture}

  \caption{A false $6$-face that is contained in a true $4$-cycle of $D$}
  \label{4-cycle}
\end{figure}
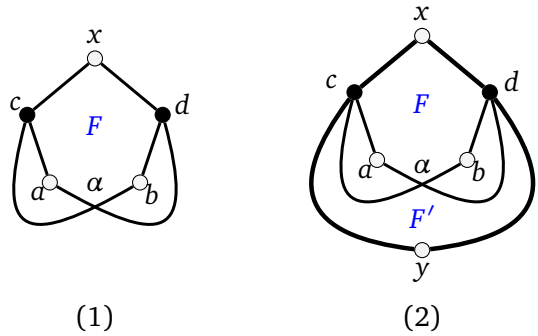

We now consider $D$ has at least one  $6$-face in $\mathcal{D}_{6}\cup \mathcal{E}_{6}$ as depicted in \ref{facein2-con} $(8)$ or $(9)$, 
\setcounter{claim}{0}
\renewcommand{\thecase}{\arabic{claim}}
\begin{claim}
Each $6$-face in $\mathcal{D}_{6}\cup \mathcal{E}_{6}$ is contained within a region bounded by a clean $4$-cycle in $D$.
\end{claim}
\begin{proof}
Let $F=a\alpha bdxca$ be a false $6$-face in $\mathcal{D}_{6}$, as illustrated in Figure \ref{4-cycle} $(1)$,
where $\alpha$ is the crossing on $\partial(F)$ and $N_{D}(\alpha)=\{a,b,c,d\}$.
Since each crossing pair forms a clean-tie in $D$, 
the regions bounded by $a\alpha bdxca$, $a\alpha ca$ and $b\alpha db$ are faces in $D$.
Consequently, all vertices in $V(G)\backslash T(F)$ lie in the exterior of the region bounded by $cxd\alpha c$.
Recall that $D$ has six types of possible faces, shown in Figure \ref{facein2-con} $(1)$-$(4)$, $(8)$ and $(9)$,
then edge segments $\alpha c$ and $\alpha d$ are incident to a false $4$-face $F'$ in $\mathcal{B}_{4}$.
If $F'=c\alpha dxc$, then $n=5$, contradicting $n\geq 6$.
Thus there exists a vertex $y$ such that $F'=c\alpha dyc$,
as shown in Figure \ref{4-cycle} $(2)$.
It follows that $\mathcal{L}_{4}=cxdyc$ is the $4$-cycle that contains the face $F$.
Since $c$ and $d$ are incident to the crossing $\alpha$,
the four edges of $\mathcal{L}_{4}$ are clean in $D$.
Therefore,  $\mathcal{L}_{4}$ is a clean $4$-cycle in $D$.
Analogously, it can be shown that each false $6$-face in $\mathcal{E}_{6}$ is contained within a region bounded by a clean $4$-cycle in $D$. 
\end{proof}

If $V(G)\backslash \{a,b,c,d,x,y\}$=$\varnothing$, then $n=6$.
In this case, $G\cong K_{2,4}$. Then $e(K_{2,4})=8 > \frac{3}{2}\times6-2=7$.
The theorem holds.

If $n\geq 7$, then $V(G)\backslash\{a,b,c,d,x,y\}\neq\varnothing$.
There exists at least one vertex in the exterior of $\mathcal{L}_{4}$.
Now we consider two subgraphs $D_{1}$ and $D_{2}$ of $D$, as shown in Figure \ref{subgraphs}.
Here, $D_{1}$ is obtained from $D$ be removing all vertices in the exterior of $\mathcal{L}_{4}$
and $D_{2}$ is obtained from $D$ be removing all vertices within $\mathcal{L}_{4}$.
Assume that $D_{1}$ and $D_{2}$ inherit the drawing of $D$. 
Note that $V(D_{1})\cap V(D_{2})=\{c,d,x,y\}$ and $E(D_{1})\cap E(D_{2})=\{cx,dx,cy,dy\}$.

\begin{claim}
Both \( D_1 \) and \( D_2 \) are 2-connected MBICP-graphs of order at least 4.
\end{claim}
\begin{proof}
Since $\mathcal{L}_{4}\subseteq D_{1}$ and $\mathcal{L}_{4}\subseteq D_{2}$, both $D_{1}$ and $D_{2}$ have order at least $4$.
Note that $D_{1}\cong K_{2,4}$ is $2$-connected and no edge can be added to $D_{1}$.
Assume that $D_{2}$ has a cut-vertex $h$ with $h\in V(\mathcal{L}_{4})$,
then there exist two components $H_{1}$ and $H_{2}$ of $D_{2}-h$ such that $V(\mathcal{L}_{4})\setminus h \subseteq V(H_{1})$.
Let $H_{1}'=D[V(H_{1})\cup \{a,b\}]$, then $H_{1}'$ and $H_{2}$ are two components of $D-h$, a contradiction to $D$ is $2$-connected.
Similarly, if $h\in V(D_{2})\setminus V(\mathcal{L}_{4})$, then $h$ would be a cut-vertex of $D$, a contradiction, thus $D_{2}$ is $2$-connected.
If an edge $e_{0}$ from the complement of $D$ can be added to $D_{2}$,
by $1$-planarity, $e_{0}$ can not be added through the interior of $\mathcal{L}_{4}$.
Thus, $e_{0}$ can only be added in the exterior of $\mathcal{L}_{4}$, which implies that $e_{0}$ can be added to $D$, a contradiction to the maximality of $D$.
Hence, both $D_{1}$ and $D_{2}$ are $2$-connected MBICP-graphs.
\end{proof}

By the induction hypothesis, we conclude that $e(D_{1})\geq\frac{3}{2}n(D_{1})-2$ and $e(D_{2})\geq\frac{3}{2}n(D_{2})-2$.
Since $n(D_{1})+n(D_{2})=n(D)+4=n(G)+4$ and $e(D_{1})+e(D_{2})=e(D)+4=e(G)+4$,
it follows that: 
\begin{align}
e(G) 
\nonumber
&= e(D_{1})+e(D_{2})-4  \\
\nonumber
&\geq \frac{3}{2} n(D_{1})-2 +\frac{3}{2} n(D_{2})-2 -4 \\
\nonumber
&= \frac{3}{2}(n(G)+4)-4-4 \\
\nonumber
&= \frac{3}{2}n-2.
\end{align}
This completes Case \ref{case1}.
\end{proof}

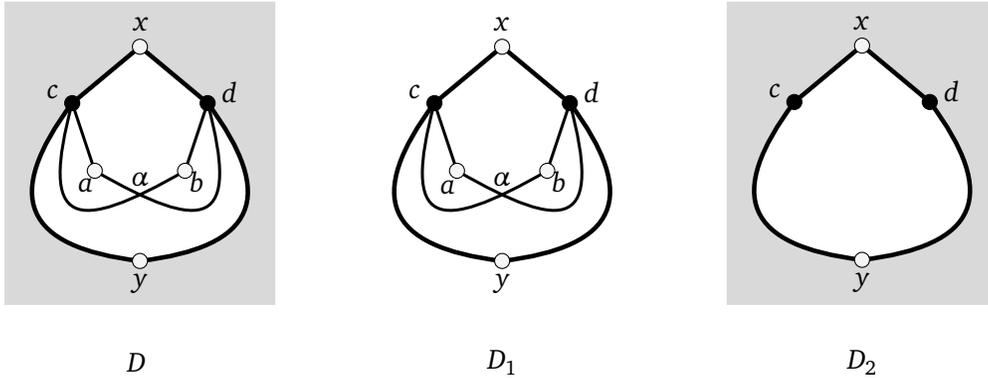
\begin{figure}[h!]
  \centering
  \begin{tikzpicture}[scale=0.6, pattern1/.style={ pattern color=black!60, pattern=north east lines}]
	\begin{pgfonlayer}{nodelayer}

		\node [style=whitenode] (5) at (8, 7.025) {};
		\node [style=blacknode] (6) at (6.5, 5.775) {};
		\node [style=blacknode] (7) at (9.5, 5.775) {};
		\node [style=whitenode] (10) at (8, 2.275) {};
		\node [style=none] (17) at (6.075, 6.025) {$c$};
		\node [style=none] (18) at (10, 6.025) {$d$};
		\node [style=none] (22) at (8, 7.525) {$x$};
		\node [style=none] (23) at (8, 1.775) {$y$};
		\node [style=none] (25) at (8, 5.775) {};
		\node [style=none] (27) at (8, 0) {$D_{2}$};
		\node [style=whitenode] (28) at (-8, 7) {};
		\node [style=blacknode] (29) at (-9.5, 5.75) {};
		\node [style=blacknode] (30) at (-6.5, 5.75) {};
		\node [style=whitenode] (31) at (-9, 4.25) {};
		\node [style=whitenode] (32) at (-7, 4.25) {};
		\node [style=whitenode] (33) at (-8, 2.25) {};
		\node [style=none] (34) at (-9.925, 6) {$c$};
		\node [style=none] (35) at (-6, 6) {$d$};
		\node [style=none] (36) at (-9.2, 3.925) {$a$};
		\node [style=none] (37) at (-6.75, 4) {$b$};
		\node [style=none] (38) at (-8, 4.075) {$\alpha$};
		\node [style=none] (39) at (-7.975, 7.5) {$x$};
		\node [style=none] (40) at (-8, 1.75) {$y$};
		\node [style=none] (41) at (-8, 5.75) {};
		\node [style=none] (42) at (-8.075, 0) {$D$};
		\node [style=none] (43) at (5, 8) {};
		\node [style=none] (44) at (5, 1.275) {};
		\node [style=none] (45) at (11, 1.275) {};
		\node [style=none] (46) at (11, 8) {};
		\node [style=whitenode] (47) at (0.025, 7) {};
		\node [style=blacknode] (48) at (-1.475, 5.75) {};
		\node [style=blacknode] (49) at (1.525, 5.75) {};
		\node [style=whitenode] (50) at (-0.975, 4.25) {};
		\node [style=whitenode] (51) at (1.025, 4.25) {};
		\node [style=whitenode] (52) at (0.025, 2.25) {};
		\node [style=none] (53) at (-1.9, 6) {$c$};
		\node [style=none] (54) at (2.025, 6) {$d$};
		\node [style=none] (55) at (-1.175, 3.9){$a$};
		\node [style=none] (56) at (1.275, 4) {$b$};
		\node [style=none] (57) at (0.025, 4.075) {$\alpha$};
		\node [style=none] (58) at (0.025, 7.5) {$x$};
		\node [style=none] (59) at (0.025, 1.75) {$y$};
		\node [style=none] (60) at (0.025, 5.75) {};
		\node [style=none] (61) at (0.025, 0) {$D_{1}$};
		\node [style=none] (62) at (-11, 8) {};
		\node [style=none] (63) at (-11, 1.275) {};
		\node [style=none] (64) at (-5, 1.275) {};
		\node [style=none] (65) at (-5, 8) {};
		\node [style=none] (66) at (-7.975, 8) {};
		\node [style=none] (67) at (-8, 1.275) {};
	\end{pgfonlayer}
	\begin{pgfonlayer}{edgelayer}
\fill[gray!30, even odd rule] (62.center) to (65.center) to (64.center) to (63.center)to (62.center)  (28.center) to (30.center)[bend left=60, looseness=1.75] to (33.center) [bend left=60, looseness=1.75] to (29.center) ;
\fill[gray!30, even odd rule] (43.center) to (46.center) to (45.center) to (44.center)to (43.center)  (5.center) to (7.center)[bend left=60, looseness=1.75] to (10.center) [bend left=60, looseness=1.75] to (6.center) ;
		\draw [style=black_bold] (5.center) to (6.center);
		\draw [style=black_bold] (5.center) to (7.center);
		\draw [style=black_bold, bend right=60, looseness=1.75] (6.center) to (10.center);
		\draw [style=black_bold, bend left=60, looseness=1.75] (7.center) to (10.center);
		\draw [style=black_bold] (28.center) to (29.center);
		\draw [style=black_bold] (28.center) to (30.center);
		\draw [style=blackedge_thick] (29.center) to (31.center);
		\draw [style=blackedge_thick] (30.center) to (32.center);
		\draw [style=blackedge_thick, in=-150, out=-105, looseness=2.50] (29.center) to (32.center);
		\draw [style=blackedge_thick, in=285, out=-30, looseness=2.50] (31.center) to (30.center);
		\draw [style=black_bold, bend right=60, looseness=1.75] (29.center) to (33.center);
		\draw [style=black_bold, bend left=60, looseness=1.75] (30.center) to (33.center);
		\draw [style=black_bold] (47.center) to (48.center);
		\draw [style=black_bold] (47.center) to (49.center);
		\draw [style=blackedge_thick] (48.center) to (50.center);
		\draw [style=blackedge_thick] (49.center) to (51.center);
		\draw [style=blackedge_thick, in=-150, out=-105, looseness=2.50] (48.center) to (51.center);
		\draw [style=blackedge_thick, in=285, out=-30, looseness=2.50] (50.center) to (49.center);
		\draw [style=black_bold, bend right=60, looseness=1.75] (48.center) to (52.center);
		\draw [style=black_bold, bend left=60, looseness=1.75] (49.center) to (52.center);
	\end{pgfonlayer}
\end{tikzpicture}

  \caption{Graph $D$ and its subgraphs $D_{1}$ and $D_{2}$}
  \label{subgraphs}
\end{figure}

For a crossing $\alpha$ in $D$, let $\mathcal{R}(\alpha)$ denote the set of regions partitioned by $T(\alpha)$,
clearly, $|\mathcal{R}(\alpha)|=3$.
Let $r(\alpha)$ denote the number of regions in $\mathcal{R}(\alpha)$ that are not faces in $D$.
Obviously, $r(\alpha)=1$ if $T(\alpha)$ is a bad tie.

\begin{case}
There exists at least one crossing $\alpha$ such that $T(\alpha)$ is not clean in $D$.
\end{case}

In this case, let $\alpha_{0}$ be a crossing in $D$ such that $T(\alpha_{0})$ is not clean and $r(\alpha_{0})$ has the largest possible value.
Clearly, $r(\alpha_{0})\geq1$.

Since $T(\alpha_{0})$ is not a clean tie, 
at least one $3$-patch of $\mathcal{R}(\alpha_{0})$ is not a false $3$-face of $D$, say $R_{1}$, as shown in Figure \ref{induction}.
Since $D$ is $2$-connected, there are at least two vertices within $R_{1}$.
We now consider two subgraphs $D_{1}$ and $D_{2}$ of $D$, as depicted in Figure \ref{induction}.
Here, $D_{1}$ is the graph obtained from $D$ by removing all vertices within $R_{2}$ and $R_{3}$,
and $D_{2}$ is the graph obtained from $D$ by removing all vertices within $R_{1}$.
Assume that $D_{1}$ and $D_{2}$ inherit the original drawing of $D$.
Note that $D_{1}\cap D_{2}=T(\alpha_{0})$.
\begin{claim}
Both \( D_1 \) and \( D_2 \) are 2-connected MBICP-graphs of order at least 4.
\end{claim}
\begin{proof}
Since $T(\alpha_{0})\subseteq D_{1}$ and $T(\alpha_{0})\subseteq D_{2}$, both $D_{1}$ and $D_{2}$ have order at least 4.
Suppose that $D_{1}$ has a cut-vertex $h$ with $h\in V(T(\alpha_{0})$.
Then there exist two components $H_{1}$ and $H_{2}$ such that $V(T_{\alpha_{0}})\setminus h\subseteq V(H_{1})$.
Let $H_{1}'=D[V(H_{1})\cup V(R_{2})\cup V(R_{3})]$, where $V(R_{2})$ and $V(R_{3})$ denote the set of vertices within $R_{2}$ and $R_{3}$ but not on their bounaries,
then $H_{1}'$ and $H_{2}$ are two components of $G-h$, which contradicts that $D$ is $2$-connected.
Similarly, if $h\in V(D_{1})\setminus V(T_{\alpha_{0}})$, then $h$ is a cut-vertex of $D$, which contradicts that $D$ is $2$-connected. Thus $D_{1}$ is $2$-connected. By a similar argument, $D_{2}$ is $2$-connected.

Now, we prove their maximality.
If an edge $e_{0}$ from the complement of $D$ can be added to $D_{1}$,
then $e_{0}$ is an edge connect a vertex $p$ within $R_{1}$ to a vertex $q$ outside $R_{1}$, where $p$ and $q$ are not vertices on the boundary of $R_{1}$;
otherwise, $e_{0}$ can be added to $D$.
In this case, $e_{0}$ would cross the unique clean edge on the boundary of $R_{1}$, a contraction to Lemma \ref{cleantie}.
Similarly, if an edge $e_{0}'$ from the complement of $D$ can be added to $D_{2}$,
then it must pass through the interior of $R_{1}$.
As a result, $e_{0}'$ is crossed at least twice, a contradiction to $1$-planarity.
Hence $D_{1}$ and $D_{2}$ are $2$-connected MBICP-graphs.
\end{proof}
\begin{subcase}\label{subcase1}
$r(\alpha_{0})\geq2$.
\end{subcase}

In this case, $D_{1}$ and $D_{2}$ are $2$-connected MBICP-graphs with $4 <n(D_{1})<n(D)$ and $4 < n(D_{2})<n(D)$.
By the induction hypothesis, $e(D_{1})\geq \frac{3}{2} n(D_{1})-2$ and $e(D_{2})\geq \frac{3}{2} n(D_{2})-2$.
Combining this with the facts that \[
n(D_{1}) + n(D_{2}) = n(G) + 4 \quad \text{and} \quad e(D_{1}) + e(D_{2}) = e(G) + 4,
\] we have
\begin{align}
e(G) 
\nonumber
&= e(D_{1})+e(D_{2})-4  \\
\nonumber
&\geq \frac{3}{2} v(D_{1})-2 +\frac{3}{2} v(D_{2})-2 -4 \\
\nonumber
&= \frac{3}{2}(n(G)+4)-4-4 \\
\nonumber
&= \frac{3}{2}n-2.
\end{align}
Thus, the conclusion holds for Subcase \ref{subcase1}.

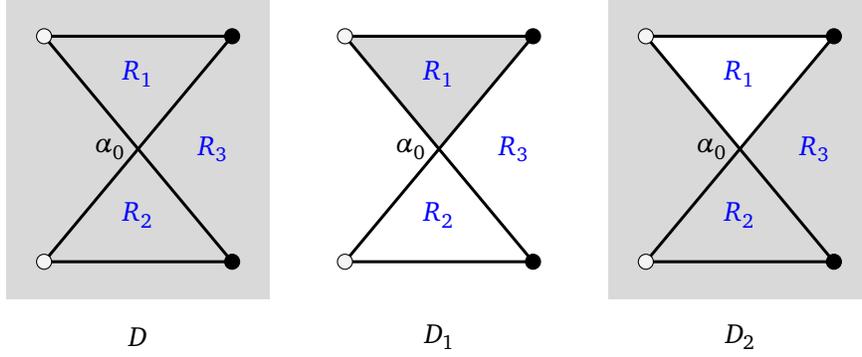
\begin{figure}[h!]
  \centering
\begin{tikzpicture}[scale=0.5, pattern1/.style={ pattern color=black!60, pattern=north east lines}]
	\begin{pgfonlayer}{nodelayer}
		\node [style=none] (0) at (-10.5, 6) {};
		\node [style=blacknode] (1) at (-5.5, 6) {};
		\node [style=whitenode] (2) at (-10.5, 0) {};
		\node [style=blacknode] (3) at (-5.5, 0) {};
		\node [style=whitenode] (4) at (-2.5, 6) {};
		\node [style=blacknode] (5) at (2.5, 6) {};
		\node [style=whitenode] (6) at (-2.5, 0) {};
		\node [style=blacknode] (7) at (2.5, 0) {};
		\node [style=whitenode] (8) at (5.5, 6) {};
		\node [style=blacknode] (9) at (10.5, 6) {};
		\node [style=whitenode] (10) at (5.5, 0) {};
		\node [style=blacknode] (11) at (10.5, 0) {};
		\node [style=none] (12) at (-8, 3) {};
		\node [style=none] (13) at (0, 3) {};
		\node [style=none] (14) at (8, 3) {};
		\node [style=none] (15) at (-8, -2) {$D$};
		\node [style=none] (16) at (0, -2) {$D_{1}$};
		\node [style=none] (17) at (8, -2) {$D_{2}$};
		\node [style=whitenode] (18) at (-10.5, 6) {};
		\node [style=none] (31) at (-8, 5) {\color{blue}$R_{1}$};
		\node [style=none] (32) at (-8, 1.25) {\color{blue}$R_{2}$};
		\node [style=none] (33) at (-6, 3) {\color{blue}$R_{3}$};
		\node [style=none] (34) at (-11.5, 7) {};
		\node [style=none] (35) at (-4.5, 7) {};
		\node [style=none] (36) at (-4.5, -1) {};
		\node [style=none] (37) at (-11.5, -1) {};
		\node [style=none] (42) at (4.5, 7) {};
		\node [style=none] (43) at (4.5, -1) {};
		\node [style=none] (44) at (11.5, -1) {};
		\node [style=none] (45) at (11.5, 7) {};
\node [style=none] (46) at (0, 5) {\color{blue}$R_{1}$};
		\node [style=none] (47) at (0, 1.25) {\color{blue}$R_{2}$};
		\node [style=none] (48) at (2, 3) {\color{blue}$R_{3}$};
		\node [style=none] (49) at (8, 5) {\color{blue}$R_{1}$};
		\node [style=none] (50) at (8, 1.25) {\color{blue}$R_{2}$};
		\node [style=none] (51) at (10, 3) {\color{blue}$R_{3}$};
		\node [style=none] (52) at (-8.75, 3) {$\alpha_{0}$};
		\node [style=none] (53) at (-0.75, 3) {$\alpha_{0}$};
		\node [style=none] (54) at (7.25, 3) {$\alpha_{0}$};
	\end{pgfonlayer}
	\begin{pgfonlayer}{edgelayer}
\fill[gray!30, even odd rule] (34.center) to (35.center) to (36.center) to (37.center) to (34.center);
\fill[gray!30, even odd rule] (4.center) to (5.center) to (13.center) to (4.center);
\fill[gray!30, even odd rule] (42.center) to (45.center) to (44.center) to (43.center) to (42.center)  (8.center) to (9.center) to (14.center) ;
		\draw [style=blackedge_thick] (0.center) to (1.center);
		\draw [style=blackedge_thick] (1.center) to (2.center);
		\draw [style=blackedge_thick] (0.center) to (3.center);
		\draw [style=blackedge_thick] (2.center) to (3.center);
		\draw [style=blackedge_thick] (4.center) to (5.center);
		\draw [style=blackedge_thick] (5.center) to (6.center);
		\draw [style=blackedge_thick] (4.center) to (7.center);
		\draw [style=blackedge_thick] (6.center) to (7.center);
		\draw [style=blackedge_thick] (8.center) to (9.center);
		\draw [style=blackedge_thick] (9.center) to (10.center);
		\draw [style=blackedge_thick] (8.center) to (11.center);
		\draw [style=blackedge_thick] (10.center) to (11.center);
		
	\end{pgfonlayer}
\end{tikzpicture}
  \caption{The graph $D$ and its subgraphs $D_{1}$ and $D_{2}$}
  \label{induction}
\end{figure}

\begin{subcase}\label{subcase2}
$r(\alpha_{0})=1$.
\end{subcase}
By the assumption of $\alpha_{0}$, the following claim holds directly.

\begin{claim}
$r(\alpha)=1$ for each crossing $\alpha$ in $D$ such that $T(\alpha)$ is not clean.
\end{claim}

In this case, $D\cong D_{1}$.
Thus,  it suffices to consider the region $R_{1}$ of $D$.
Let $cr(D)=c$. 
Suppose there are $k$  bad ties, 
and $c-k$  clean ties, where $1\leq k \leq c$.
We have the  following claim.

\begin{claim}
$|\mathcal{F}_{5}|+2|\mathcal{F}_{6}|=2c+k$.
\end{claim}
\begin{proof}
By Proposition  \ref{tie and face} (i) and (ii),
we have $|\mathcal{F}_{3}|=k+2(c-k)=2c-k$ and $|\mathcal{B}_{6}|=k$.
Now, we consider a bipartite graph $H=(X,Y)$, where $X$ is the set of bad ties in $D$ and $Y$ is the face set of $\mathcal{F}_{5}\cup \mathcal{A}_{6}$. 
By Proposition  \ref{tie and face} (iii), each false $5$-face in $\mathcal{F}_{5}$ is incident to exactly one bad tie,
and each face in $\mathcal{A}_{6}$ is incident to exactly two bad ties.
Thus, $e(H)=|\mathcal{F}_{5}|+2|\mathcal{A}_{6}|$.
By Proposition \ref{tie and face} (i), each bad tie is incident to a face in $\mathcal{\mathcal{F}}_{5}\cup \mathcal{\mathcal{A}}_{6}$.
Since $\mathcal{F}_{5} \cap\mathcal{A}_{6}=\varnothing$, we have $e(H)=k$,
implying $|\mathcal{F}_{5}|+2|\mathcal{A}_{6}|=k$.
Moreover, by Proposition  \ref{tie and face} (iv), each face in $\mathcal{D}_{6}\cup \mathcal{E}_{6}$ is incident to exactly one clean tie.
By Proposition \ref{tie and face} (ii), each clean tie is incident to at most one face in $\mathcal{D}_{6}\cup \mathcal{E}_{6}$.
Since $\mathcal{D}_{6} \cap\mathcal{E}_{6}=\varnothing$, $|\mathcal{D}_{6}|+|\mathcal{E}_{6}|\leq c-k$.
Combining these results, we have
\begin{align}
|\mathcal{F}_{5}|+2|\mathcal{F}_{6}| 
\nonumber
&= |\mathcal{F}_{5}|+2|\mathcal{A}_{6}|+2|\mathcal{B}_{6}|+2|\mathcal{D}_{6}|+2|\mathcal{E}_{6}|  \\
\nonumber
&\leq (|\mathcal{F}_{5}|+2|\mathcal{A}_{6}|)+2k+2c-2k \\
\nonumber
&= k+2k+2c-2k \\
\nonumber
&= 2c+k.
\end{align}
Hence, the claim holds.
\end{proof}

Let $D^{\times}$ be the associated graph of $D$.
Then $$n(D^{\times})=n(D)+c, e(D^{\times})=e(D)+2c$$ and $$f(D^{\times})=\sum_{t=3}^{6}f_{t}, $$
where $f(D^{\times})$ denote the number of faces in $D^{\times}$ and $f_{t}$ denote the number of faces of size $t$ in $D^{\times}$.
By Proposition \ref{seven}, $3\leq t\leq 6$.
Applying Euler's formula $n(D^{\times})-e(D^{\times})+f(D^{\times})=2$,
we obtain $$(n(D)+c)-(e(D)+2c)+\sum_{t=3}^{6}f_{t}=n(D)+c-e(D)-2c+f_{3}+f_{4}+f_{5}+f_{6}=2,$$
which implies $f_{4}=2+e(D)+c-n(D)-f_{3}-f_{5}-f_{6}$.

From the handshaking lemma, we have $2e(D^{\times})= \sum _{t=3}^{6}tf_{t}$, 
Furthermore, we have 
\begin{align*}
2e(D)+4c
\nonumber
&= 3f_{3}+4f_{4}+5f_{5}+6f_{6}  \\
\nonumber
&= 3f_{3}+(8+4e(D)+4c-4n(D)-4f_{3}-4f_{5}-4f_{6})+5f_{5}+6f_{6},
\end{align*}
implying that 
\begin{align*}
2e(D)
\nonumber
&= 4n(D)-8+f_{3}-f_{5}-2f_{6}  \\
\nonumber
&= 4n(D)-8+2c-k-2c-k,
\end{align*}
and thus $e(D)=2n(D)-4-k$.

Since $k\leq c$, we have $e(D)\geq 2n(D)-4-c$.
By Lemma \ref{lem:zhang}, $c\leq\frac{n(G)}{4}$, 
it follows that 
\begin{align*}
e(G)&=e(D)\geq 2n(G)-4-\frac{n(G)}{4}\\
&=\frac{7}{4}n-4.
\end{align*}

Then, $e(G)\geq \lceil\frac{7}{4}n(G)-4 \rceil \geq\frac{3}{2}n-2$ holds for $n\geq6$.
Thus, the result also holds for Subcase \ref{subcase2}.
\end{proof}

Now, we show that the lower bound of Theorem \ref{2-connected} is sharp by the following lemma.

\begin{lemma}\label{2-connectedtheorem}
There exist infinitely many $2$-connected MBICP-graphs $H_{k}$ with $n(H_{k})=4k$ ($k\geq1$)  and $e(H_{k})=\frac{3}{2}n(H_{k})-2$.
\end{lemma}
\begin{proof}
First, we present two configurations $W_{1}$ and $W_{2}$, as shown in Figure \ref{2-connectedgraph} $(1)$ and $(2)$,
where the dashed lines represent the boundary of a false $3$-face.

For $n=4$, let $H_{1}$ be the complete bipartite graph $K_{2,2}$ with two crossing edges as shown in Figure \ref{2-connectedgraph} $(3)$.
Note that $H_{1}$ contains exactly two false $3$-faces, $F_{1}$ and $F_{2}$.
Then we select an arbitrary false $3$-face of $H_{1}$ and insert a configuration $W_{i}$ $(i\in\{1,2\})$ into it to construct the graph $H_{2}$,
where the selection of $W_{i}$ depends on which false $3$-face we choose.
For instance, the graph $H_{2}$ shown in Figure \ref{2-connectedgraph} $(4)$ is obtained by inserting $W_{1}$ into the face $F_{2}$ of $H_{1}$.
Observe that $H_{2}$ also has exactly two false $3$-faces.

Similarly, for $k\geq3$, the graph $H_{k}$ is constructed by choosing an arbitrary false $3$-face of $H_{k-1}$ 
and inserting the corresponding configuration $W_{i}$ $(i\in\{1,2\})$ into it.
We can check that  $H_{k}$ is a MBICP-graph with exactly two false $3$-faces.
Moreover, for each $k\geq1$, $H_{k}$ is $2$-connected, as it contains no cut-vertex.
For $k=1$, we have $n(H_{1})=4$ and $e(H_{1})=4$.
For $k\geq2$, We add four new vertices and six new edges from $H_{k}$ to $H_{k-1}$.
It implies that $n(H_{k})=4k$ and $e(H_{k})=4+6(k-1)=6k-2$.
Thus, for each $k\geq1$, $e(H_{k})=\frac{3}{2}n(H_{k})-2$.
Hence we complete the proof.
\end{proof}

\begin{figure}[h!]
  \centering
\begin{tikzpicture}[scale=0.6]
	\begin{pgfonlayer}{nodelayer}
		\node [style=whitenode] (0) at (-12, 0) {};
		\node [style=blacknode] (1) at (-7, 0) {};
		\node [style=none] (2) at (-9.5, 4) {};
		\node [style=blacknode] (3) at (-10, 1.75) {};
		\node [style=blacknode] (4) at (-10, 1) {};
		\node [style=whitenode] (5) at (-9, 1.75) {};
		\node [style=whitenode] (6) at (-9, 1) {};
		\node [style=blacknode] (7) at (-5, 0) {};
		\node [style=whitenode] (8) at (0, 0) {};
		\node [style=none] (9) at (-2.5, 4) {};
		\node [style=whitenode] (10) at (-3, 1.75) {};
		\node [style=whitenode] (11) at (-3, 1) {};
		\node [style=blacknode] (12) at (-2, 1.75) {};
		\node [style=blacknode] (13) at (-2, 1) {};
		\node [style=whitenode] (14) at (2, 5) {};
		\node [style=whitenode] (15) at (2, 0) {};
		\node [style=blacknode] (16) at (6, 5) {};
		\node [style=blacknode] (17) at (6, 0) {};
		\node [style=whitenode] (26) at (9, 5) {};
		\node [style=blacknode] (27) at (14, 5) {};
		\node [style=whitenode] (28) at (9, 0) {};
		\node [style=blacknode] (29) at (14, 0) {};
		\node [style=blacknode] (30) at (11, 2.1) {};
		\node [style=whitenode] (31) at (12, 2.1) {};
		\node [style=blacknode] (32) at (11, 1.35) {};
		\node [style=whitenode] (33) at (12, 1.35) {};
		\node [style=none] (34) at (4, 4) {\color{blue}$F_{1}$};
		\node [style=none] (35) at (4, 1) {\color{blue}$F_{2}$};
		\node [style=none] (36) at (-9.5, -1) {};
		\node [style=none] (37) at (-2.5, -1) {$(2)$ $W_{2}$};
		\node [style=none] (38) at (4, -1) {$(3)$ $H_{1}$};
		\node [style=none] (39) at (11.5, -1) {$(4)$ $H_{2}$};
		\node [style=none] (40) at (-9.5, -1) {$(1)$ $W_{1}$};
		\node [style=none] (41) at (-9, 4.75) {};
		\node [style=none] (42) at (-10, 4.75) {};
		\node [style=none] (43) at (-2, 4.75) {};
		\node [style=none] (44) at (-3, 4.75) {};
	\end{pgfonlayer}
	\begin{pgfonlayer}{edgelayer}
		\draw [dash pattern=on 2pt off 2pt] (0) to (41.center);
		\draw [dash pattern=on 2pt off 2pt] (1) to (42.center);
		\draw [dash pattern=on 2pt off 2pt] (43.center) to (7);
		\draw [dash pattern=on 2pt off 2pt] (44.center) to (8);
		\draw [dash pattern=on 2pt off 2pt] (0) to (1);
		\draw [style={blackedge_thick}] (3) to (6);
		\draw [style={blackedge_thick}] (5) to (4);
		\draw [style={blackedge_thick}] (4) to (6);
		\draw [style={blackedge_thick}, bend right=135, looseness=5.25] (3) to (5);
		\draw [style={blackedge_thick}] (3) to (0);
		\draw [style={blackedge_thick}] (5) to (1);
		\draw [dash pattern=on 2pt off 2pt] (7) to (8);
		\draw [style={blackedge_thick}] (10) to (13);
		\draw [style={blackedge_thick}] (12) to (11);
		\draw [style={blackedge_thick}] (11) to (13);
		\draw [style={blackedge_thick}, bend right=135, looseness=5.25] (10) to (12);
		\draw [style={blackedge_thick}] (10) to (7);
		\draw [style={blackedge_thick}] (12) to (8);
		\draw [style={blackedge_thick}] (16) to (15);
		\draw [style={blackedge_thick}] (14) to (17);
		\draw [style={blackedge_thick}] (14) to (16);
		\draw [style={blackedge_thick}] (15) to (17);
		\draw [style={blackedge_thick}](26) to (27);
		\draw [style={blackedge_thick}, bend right] (27) to (28);
		\draw [style={blackedge_thick}] (28) to (29);
		\draw [style={blackedge_thick}, bend left] (26) to (29);
		\draw [style={blackedge_thick}] (30) to (33);
		\draw [style={blackedge_thick}] (31) to (32);
		\draw [style={blackedge_thick}] (32) to (33);
		\draw [style={blackedge_thick}, bend right=135, looseness=6.00] (30) to (31);
		\draw [style={blackedge_thick}] (30) to (28);
		\draw [style={blackedge_thick}] (31) to (29);
	\end{pgfonlayer}
\end{tikzpicture}
  \caption{The construction of extremal graphs with $\frac{3}{2}n-2$ edges.}
  \label{2-connectedgraph}
\end{figure}
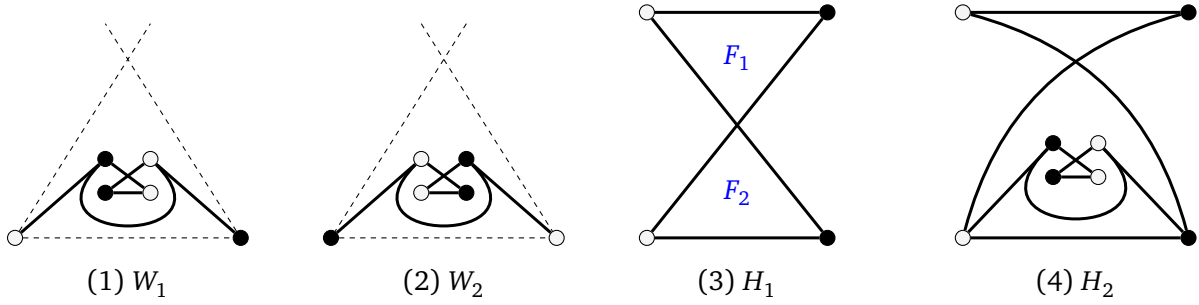

\subsection{Proof of Theorem \ref{3-connected}}
\begin{proof}
Let $D$ be the corresponding maximal bipartite IC-planar drawing of $G$.
Since $G$ is 3-connected,
by Proposition \ref{3-connected property} (iv), $D_{p}$ is a quadrangulation.
Thus, $e(D_{p})=2n(D_{p})-4=2n(G)-4$.
By Corollary \ref{3-con cr=0}, $cr(D)\geq1$.
Then $e(G)=e(D_{p})+cr(D)\geq 2n(G)-4+1=2n(G)-3$.
\end{proof}

The extremal construction below shows that the bound in Theorem \ref{3-connected} is sharp.

\begin{lemma}\label{3-connectedtheorem}
For any $k\geq3$, there exists a  $3$-connected MBICP-graph with $n=2k$ vertices and  $2n-3$ edges.
\end{lemma}
\begin{proof}
To begin with, an even cycle $C$ of length $n-2$ is constructed, containing $\frac{n}{2}-1$ white vertices and $\frac{n}{2}-1$ black vertices.
Let $C$ be drawn in the plane without any crossings.
 Subsequently, a black vertex $b_{\frac{n}{2}}$ is inserted inside $C$, and a white vertex $w_{\frac{n}{2}}$ is placed outside $C$.
Next, connect $b_{\frac{n}{2}}$ to all white vertices of $V(C)$ within $C$, 
and connect $w_{\frac{n}{2}}$ to all black vertices of $V(C)$ in the exterior of $C$.
The resulting graph is known as a pseudo double wheel graph, which we denote by $W_{n}$.
Note that $W_{n}$ is a quadrangulation of order $n$.
Thus $e(W_{n})=2n-4$.
Finally, we add an edge $b_{\frac{n}{2}}w_{\frac{n}{2}}$ to $W_{n}$ to obtain graph $G_{n}$,
as shown in Figure \ref{3-connectedgraph}.

Observe that the edge $b_{\frac{n}{2}}w_{\frac{n}{2}}$ crosses with an edge of $E(C)$.
Consequently, vertices $b_{\frac{n}{2}}$ and $w_{\frac{n}{2}}$ are incident with a crossing.
Thus it is not difficult to see that no more edges can be added to $G_{n}$ without violating the assumption of $G$.
Hence, $G_{n}$ is a MBICP-graph.

Now, we prove $G_{n}$ is $3$-connected.
Since $W_{n}\subseteq G_{n}$ and $W_{n}$ is $3$-connected for $n\geq 8$, 
it follows that $G_{n}$ is $3$-connected for $n\geq8$. 
For $n = 6$, one has $G_6 \cong K_{3,3}$, which is $3$-connected.
Hence, $G_{n}$ is $3$-connected.
Moreover, $n(G_{n})=n$ and $e(G_{n})=e(W_{n})+1=2n-3$.
\end{proof}

\begin{figure}[h!]
  \centering
\begin{tikzpicture}[scale=0.65]
	\begin{pgfonlayer}{nodelayer}
		\node [style=blacknode] (4) at (0, 0) {};
		\node [style=blacknode] (5) at (0, 3) {};
		\node [style=blacknode] (6) at (0, -3) {};
		\node [style=blacknode] (7) at (-3.025, 0) {};
		\node [style=blacknode] (8) at (3, 0) {};
		\node [style=whitenode] (9) at (-2.25, 2.025) {};
		\node [style=whitenode] (10) at (2.25, 2) {};
		\node [style=whitenode] (11) at (2.25, -2) {};
		\node [style=whitenode] (12) at (-2.3, -2.05) {};
		\node [style=whitenode] (13) at (0, 6) {};
		\node [style=none] (14) at (0.5, 3.5) {$b_{1}$};
		\node [style=none] (15) at (2.725, 2.325) {$w_{1}$};
		\node [style=none] (16) at (3.5, 0) {$b_{2}$};
		\node [style=none] (17) at (3, -1.75) {$w_{2}$};
		\node [style=none] (18) at (0, -3.5) {$b_{3}$};
		\node [style=none] (19) at (-2.75, -2.5) {$w_{3}$};
		\node [style=none] (20) at (-3.825, 0) {$b_{\frac{n}{2}-1}$};
		\node [style=none] (21) at (-2.3, 2.7) {$w_{\frac{n}{2}-1}$};
		\node [style=none] (22) at (0.75, 0) {$b_{\frac{n}{2}}$};
		\node [style=none] (23) at (0, 6.75) {$w_{\frac{n}{2}}$};
		\node [style=blacknode_v1] (24) at (-2.425, -0.4) {};
		\node [style=blacknode_v1] (25) at (-2.35, -0.8) {};
		\node [style=blacknode_v1] (26) at (-2.125, -1.125) {};
		\node [style=blacknode_v1] (27) at (-3.45, -1.2) {};
		\node [style=blacknode_v1] (28) at (-3.35, -1.55) {};
		\node [style=blacknode_v1] (29) at (-3.1, -1.85)  {};
	\end{pgfonlayer}
	\begin{pgfonlayer}{edgelayer}
		\draw [style=blackedge_thick, bend left=45] (5) to (8);
		\draw [style=blackedge_thick, bend left=45] (8) to (6);
		\draw [style=blackedge_thick, bend right=45] (5) to (7);
		\draw [style=blackedge_thick] (4) to (9.center);
		\draw [style=blackedge_thick] (4) to (10.center);
		\draw [style=blackedge_thick] (4) to (11.center);
		\draw [style=blackedge_thick] (4) to (12.center);
		\draw [style=blackedge_thick] (13.center) to (5);
		\draw [style=blackedge_thick, bend left=45, looseness=1.25] (13.center) to (8);
		\draw [style=blackedge_thick, in=105, out=-165] (13.center) to (7);
		\draw [style=blackedge_thick, in=-15, out=0, looseness=2.00] (13.center) to (6);
		\draw [style=black_bold, bend right] (13.center) to (4);
		\draw [dash pattern=on 2pt off 2pt, bend right=15, looseness=1.25] (7) to (12.center);
		\draw [style=blackedge, in=180, out=-45, looseness=0.75]  (12.center) to (6);
	\end{pgfonlayer}
\end{tikzpicture}
  \caption{The extremal graph $G_{n}$ with $2n-3$ edges.}
  \label{3-connectedgraph}
\end{figure}
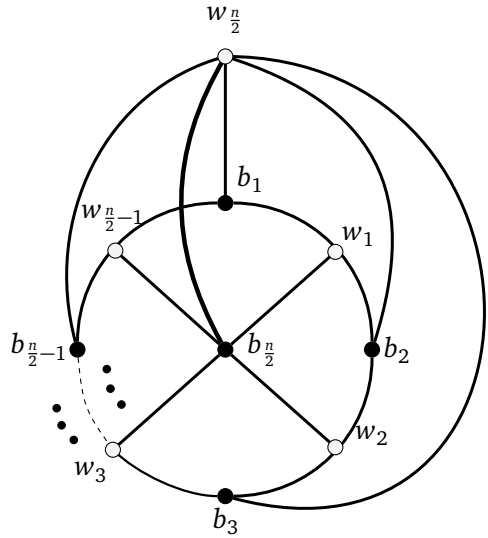

\section{Concluding remarks and open problems}\label{further study}
\noindent In the previous section, we determined the minimum size of MBIC-graphs with connectivity 2 and 3.
However, Problem \ref{prob:1} for connectivity 4 has not yet been solved. Nevertheless, we are able to conjecture a bound.

Given a graph $G$=$(V(G),E(G))$ and two subsets $S, T\subset V(G)$ such that $S\cap T=\varnothing$,
denote by $E_{G}(S,T)$ (or simply $E(S,T)$) the set of edges in $G$ with one end in $S$ and the other end in $T$,
and let $e_{G}(S,T)$ (or simply $e(S,T)$) denote the cardinality of $E_{G}(S,T)$.

\begin{theorem}\label{4-connected extremal graph}
For $k\geq1$, there exist infinitely many $4$-connected MBIC-graphs $G_{k}$ with $n(G_{k})=24k+4$ and $e(G_{k})=\frac{13n(G_{k})}{6}-\frac{14}{3}$.
\end{theorem}
\begin{proof}
First, we present the construction of the graph $G_{k}$.

Let $C_{1}=a_{1}a_{2}a_{3}a_{4}a_{1}$, 
$C_{2}=b_{1}b_{2}b_{3}\cdots b_{11}b_{12}b_{1}$
and $C_{3}=c_{1}c_{2}c_{3}\cdots c_{11}c_{12}c_{1}$
denote three distinct cycles,
where vertices with odd indices are white and vertices with even indices are black.
These cycles are referred to as the {\it base cycles} of $G_{1}$.
 
 Let
 \[
E(C_{1},C_{2}) = \{a_{1}b_{2}, a_{1}b_{12}, a_{2}b_{3}, a_{2}b_{5}, a_{3}b_{6}, a_{3}b_{8}, a_{4}b_{9}, a_{4}b_{11}\},
\]
\begin{align*}
E(C_{2},C_{3}) = \{ & b_{1}c_{2}, b_{1}c_{12}, b_{2}c_{3}, b_{3}c_{2}, b_{4}c_{3}, b_{4}c_{5}, b_{5}c_{6}, 
  b_{6}c_{5}, b_{7}c_{6}, \\ & b_{7}c_{8}, b_{8}c_{9}, b_{9}c_{8}, b_{10}c_{9}, b_{10}c_{11}, b_{11}c_{12}, b_{12}c_{11} \}
\end{align*}
and
\[
E(C_{3},C_{3}) = \{c_{1}c_{4}, c_{4}c_{7}, c_{7}c_{10}, c_{10}c_{1}\}.
\]

For $k=1$,
the graph $G_{1}=(V(G_{1}),E(G_{1}))$ is defined as illustrated in Figure \ref{graph Gk} $(1)$,
where $V(G_{1})=V(C_{1})\cup V(C_{2}) \cup V(C_{3})$ and $E(G_{1})=E(C_{1})\cup E(C_{2})\cup E(C_{3})\cup E(C_{3},C_{3}) \cup E(C_{1},C_{2})\cup E(C_{2},C_{3})$.
Clearly $G_{1}$ is a bipartite IC-plane graph with $n(G_{1})=24k+4=24+4=28$ and $e(G_{1})=\frac{13\times 28}{6}-\frac{14}{3}=56$.

For $k\geq2$, we construct $G_{k}$ recursively via the illustration given in Figure \ref{graph Gk} $(2)$:
put the entire graph $G_{k-1}$ inside the innermost $4$-face of $G_{1}$ bounded by $a_{1}a_{2}a_{3}a_{4}a_{1}$,
and merge the boundary of the outermost $4$-cycle of $G_{k-1}$ with the boundary of the innermost $4$-cycle of $G_{1}$.
The instance for $k=2$ is shown in Figure \ref{G2}.
One can verify that each iteration from $G_{k-1}$ to $G_{k}$ adds two base cycles of length twelve (i.e., $24$ vertices) and $52$ edges.
It follows that $n(G_{k})=24k+4$, $e(G_{k})=56+52(k-1)=52k+4$, and $G_{k}$ has $2k+1$ base cycles.
Therefore, $e(G_{k})=\frac{13n(G_{k})}{6}-\frac{14}{3}$.
For each $k\geq1$, $G_{k}$ is a MBICP-graph, 
since no edges from the complement of $G_{k}$ can be added without violating the assumption of $G_{k}$.
Thus, it remains to prove that $\kappa(G_k) = 4$, with the tedious details given in Appendix I.
\end{proof}

\begin{figure}[h!]
  \centering
\begin{tikzpicture}[scale=0.4]
	\begin{pgfonlayer}{nodelayer}
		\node [style=whitenode] (0) at (6.5, 1.5) {};
		\node [style=whitenode] (1) at (8.5, -0.5) {};
		\node [style=blacknode] (2) at (8.5, 1.5) {};
		\node [style=blacknode] (3) at (6.5, -0.5) {};
		\node [style=whitenode] (4) at (4.5, -0.5) {};
		\node [style=whitenode] (5) at (10.5, 1.5) {};
		\node [style=whitenode] (7) at (8.5, 3.5) {};
		\node [style=blacknode] (8) at (6.5, 3.5) {};
		\node [style=blacknode] (9) at (10.5, 3.5) {};
		\node [style=blacknode] (10) at (10.5, -0.5) {};
		\node [style=blacknode] (11) at (4.5, -2.5) {};
		\node [style=blacknode] (12) at (4.5, 1.5) {};
		\node [style=whitenode] (13) at (10.5, -2.5) {};
		\node [style=whitenode] (14) at (4.5, 3.5) {};
		\node [style=whitenode] (15) at (6.5, -2.5) {};
		\node [style=blacknode] (16) at (8.5, -2.5) {};
		\node [style=blacknode] (17) at (6.5, 5.5) {};
		\node [style=blacknode] (18) at (12.5, 5.5) {};
		\node [style=blacknode] (19) at (2.5, -4.5) {};
		\node [style=blacknode] (20) at (2.5, 1.5) {};
		\node [style=blacknode] (21) at (12.5, -0.5) {};
		\node [style=whitenode] (22) at (8.5, 5.5) {};
		\node [style=whitenode] (23) at (2.5, 5.5) {};
		\node [style=whitenode] (24) at (12.5, -4.5) {};
		\node [style=whitenode] (25) at (12.5, 1.5) {};
		\node [style=whitenode] (26) at (2.5, -0.5) {};
		\node [style=whitenode] (27) at (6.5, -4.5) {};
		\node [style=blacknode] (28) at (8.5, -4.5) {};
		\node [style=none] (29) at (7.5, 0.5) {$G_{k-1}$};
		\node [style=whitenode] (30) at (-8.5, 1.5) {};
		\node [style=whitenode] (31) at (-6.5, -0.5) {};
		\node [style=blacknode] (32) at (-6.5, 1.5) {};
		\node [style=blacknode] (33) at (-8.5, -0.5) {};
		\node [style=whitenode] (34) at (-10.5, -0.5) {};
		\node [style=whitenode] (35) at (-4.5, 1.5) {};
		\node [style=whitenode] (36) at (-6.5, 3.5) {};
		\node [style=blacknode] (37) at (-8.5, 3.5) {};
		\node [style=blacknode] (38) at (-4.5, 3.5) {};
		\node [style=blacknode] (39) at (-4.5, -0.5) {};
		\node [style=blacknode] (40) at (-10.5, -2.5) {};
		\node [style=blacknode] (41) at (-10.5, 1.5) {};
		\node [style=whitenode] (42) at (-4.5, -2.5) {};
		\node [style=whitenode] (43) at (-10.5, 3.5) {};
		\node [style=whitenode] (44) at (-8.5, -2.5) {};
		\node [style=blacknode] (45) at (-6.5, -2.5) {};
		\node [style=blacknode] (46) at (-8.5, 5.5) {};
		\node [style=blacknode] (47) at (-2.5, 5.5) {};
		\node [style=blacknode] (48) at (-12.5, -4.5) {};
		\node [style=blacknode] (49) at (-12.5, 1.5) {};
		\node [style=blacknode] (50) at (-2.5, -0.5) {};
		\node [style=whitenode ] (51) at (-6.5, 5.5) {};
		\node [style=whitenode] (52) at (-12.5, 5.5) {};
		\node [style=whitenode] (53) at (-2.5, -4.5) {};
		\node [style=whitenode] (54) at (-2.5, 1.5) {};
		\node [style=whitenode] (55) at (-12.5, -0.5) {};
		\node [style=whitenode] (56) at (-8.5, -4.5) {};
		\node [style=blacknode] (57) at (-6.5, -4.5) {};
		\node [style=none] (59) at (-7.5, -7) {$G_{1}$};
		\node [style=none] (60) at (7.5, -7) {$G_{k}$};

	\end{pgfonlayer}
	\begin{pgfonlayer}{edgelayer}
		\draw [style=rededge_thick] (0.center) to (2);
		\draw [style=rededge_thick] (2) to (1.center);
		\draw [style=rededge_thick] (0.center) to (3);
		\draw [style=rededge_thick] (3) to (1.center);
		\draw [style=grayedge] (8) to (0.center);
		\draw [style=grayedge] (7.center) to (2);
		\draw [style=blue_thick] (8) to (7.center);
		\draw [style=blue_thick] (14.center) to (8);
		\draw [style=blue_thick] (14.center) to (12);
		\draw [style=grayedge] (12) to (0.center);
		\draw [style=blue_thick] (7.center) to (9);
		\draw [style=blue_thick] (9) to (5.center);
		\draw [style=grayedge] (2) to (5.center);
		\draw [style=blue_thick] (5.center) to (10);
		\draw [style=grayedge] (1.center) to (10);
		\draw [style=blue_thick] (12) to (4.center);
		\draw [style=grayedge] (4.center) to (3);
		\draw [style=blue_thick] (4.center) to (11);
		\draw [style=blue_thick] (11) to (15.center);
		\draw [style=grayedge] (3) to (15.center);
		\draw [style=blue_thick] (15.center) to (16);
		\draw [style=grayedge] (1.center) to (16);
		\draw [style=blue_thick] (16) to (13.center);
		\draw [style=blue_thick] (10) to (13.center);
		\draw [style=grayedge] (14.center) to (17);
		\draw [style=grayedge] (20) to (4.center);
		\draw [style=grayedge] (12) to (26.center);
		\draw [style=grayedge] (14.center) to (20);
		\draw [style=blue_thick] (23.center) to (20);
		\draw [style=blue_thick] (23.center) to (17);
		\draw [style=blue_thick] (17) to (22.center);
		\draw [style=blue_thick] (22.center) to (18);
		\draw [style=blue_thick] (18) to (25.center);
		\draw [style=blue_thick] (25.center) to (21);
		\draw [style=blue_thick] (21) to (24.center);
		\draw [style=grayedge] (15.center) to (28);
		\draw [style=grayedge] (16) to (27.center);
		\draw [style=blue_thick] (28) to (24.center);
		\draw [style=blue_thick] (27.center) to (28);
		\draw [style=blue_thick] (19) to (27.center);
		\draw [style=grayedge] (11) to (27.center);
		\draw [style=grayedge] (11) to (26.center);
		\draw [style=blue_thick] (20) to (26.center);
		\draw [style=grayedge] (10) to (25.center);
		\draw [style=grayedge] (5.center) to (21);
		\draw [style=grayedge] (9) to (25.center);
		\draw [style=grayedge] (9) to (22.center);
		\draw [style=grayedge] (17) to (7.center);
		\draw [style=grayedge] (8) to (22.center);
		\draw [style=grayedge] (21) to (13.center);
		\draw [style=grayedge] (13.center) to (28);
		\draw [style=blue_thick] (26.center) to (19);
		\draw [style=grayedge, bend left] (23.center) to (18);
		\draw [style=grayedge, bend left] (18) to (24.center);
		\draw [style=grayedge, bend right] (23.center) to (19);
		\draw [style=grayedge, bend right] (19) to (24.center);
		\draw [style=blue_thick] (30.center) to (32);
		\draw [style=blue_thick] (32) to (31.center);
		\draw [style=blue_thick] (30.center) to (33);
		\draw [style=blue_thick] (33) to (31.center);
		\draw [style=grayedge] (37) to (30.center);
		\draw [style=grayedge] (36.center) to (32);
		\draw [style=blue_thick] (37) to (36.center);
		\draw [style=blue_thick] (43.center) to (37);
		\draw [style=blue_thick] (43.center) to (41);
		\draw [style=grayedge] (41) to (30.center);
		\draw [style=blue_thick] (36.center) to (38);
		\draw [style=blue_thick] (38) to (35.center);
		\draw [style=grayedge] (32) to (35.center);
		\draw [style=blue_thick] (35.center) to (39);
		\draw [style=grayedge] (31.center) to (39);
		\draw [style=blue_thick] (41) to (34.center);
		\draw [style=grayedge] (34.center) to (33);
		\draw [style=blue_thick] (34.center) to (40);
		\draw [style=blue_thick] (40) to (44.center);
		\draw [style=grayedge] (33) to (44.center);
		\draw [style=blue_thick] (44.center) to (45);
		\draw [style=grayedge] (31.center) to (45);
		\draw [style=blue_thick] (45) to (42.center);
		\draw [style=blue_thick] (39) to (42.center);
		\draw [style=grayedge] (43.center) to (46);
		\draw [style=grayedge] (49) to (34.center);
		\draw [style=grayedge] (41) to (55.center);
		\draw [style=grayedge] (43.center) to (49);
		\draw [style=blue_thick] (52.center) to (49);
		\draw [style=blue_thick] (52.center) to (46);
		\draw [style=blue_thick] (46) to (51.center);
		\draw [style=blue_thick] (51.center) to (47);
		\draw [style=blue_thick] (47) to (54.center);
		\draw [style=blue_thick] (54.center) to (50);
		\draw [style=blue_thick] (50) to (53.center);
		\draw [style=grayedge] (44.center) to (57);
		\draw [style=grayedge] (45) to (56.center);
		\draw [style=blue_thick] (57) to (53.center);
		\draw [style=blue_thick] (56.center) to (57);
		\draw [style=blue_thick] (48) to (56.center);
		\draw [style=grayedge] (40) to (56.center);
		\draw [style=grayedge] (40) to (55.center);
		\draw [style=blue_thick] (49) to (55.center);
		\draw [style=grayedge] (39) to (54.center);
		\draw [style=grayedge] (35.center) to (50);
		\draw [style=grayedge] (38) to (54.center);
		\draw [style=grayedge] (38) to (51.center);
		\draw [style=grayedge] (46) to (36.center);
		\draw [style=grayedge] (37) to (51.center);
		\draw [style=grayedge] (50) to (42.center);
		\draw [style=grayedge] (42.center) to (57);
		\draw [style=blue_thick] (55.center) to (48);
		\draw [style=grayedge, bend left] (52.center) to (47);
		\draw [style=grayedge, bend left] (47) to (53.center);
		\draw [style=grayedge, bend right] (52.center) to (48);
		\draw [style=grayedge, bend right] (48) to (53.center);
	\end{pgfonlayer}
\end{tikzpicture}
  \caption{Graph $G_{1}$ and the construction of graph $G_{k}$.}
  \label{graph Gk}
\end{figure}
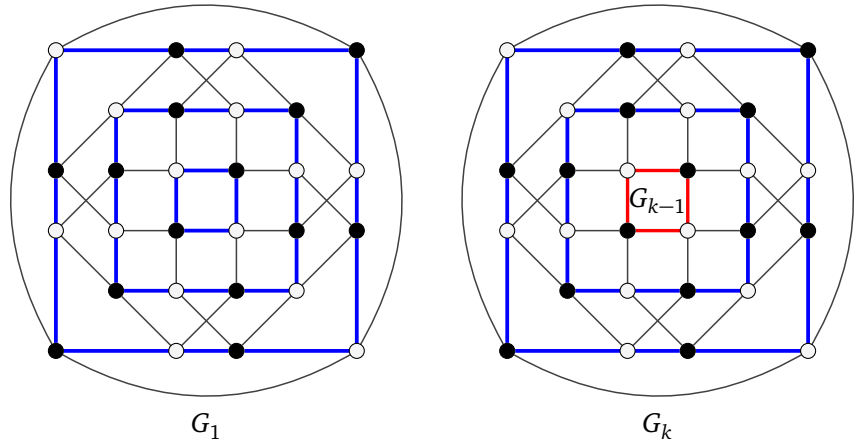

\begin{figure}[h!]
  \centering
\begin{tikzpicture}[scale=0.5]
	\begin{pgfonlayer}{nodelayer}
		\node [style=whitenode] (0) at (-0.5, 0.5) {};
		\node [style=whitenode] (1) at (0.5, -0.5) {};
		\node [style=blacknode] (2) at (0.5, 0.5) {};
		\node [style=blacknode] (3) at (-0.5, -0.5) {};
		\node [style=whitenode] (4) at (0.5, 1.5) {};
		\node [style=whitenode] (5) at (-0.5, -1.5) {};
		\node [style=blacknode] (6) at (-0.5, 1.5) {};
		\node [style=blacknode] (7) at (0.5, -1.5) {};
		\node [style=whitenode] (8) at (1.5, 0.5) {};
		\node [style=whitenode] (9) at (-1.5, -0.5) {};
		\node [style=blacknode] (10) at (-1.5, 0.5) {};
		\node [style=blacknode] (11) at (1.5, -0.5) {};
		\node [style=blacknode] (12) at (1.5, 1.5) {};
		\node [style=blacknode] (13) at (-1.5, -1.5) {};
		\node [style=whitenode] (14) at (1.5, -1.5) {};
		\node [style=whitenode] (15) at (-1.5, 1.5) {};
		\node [style=whitenode] (16) at (0.5, 2.5) {};
		\node [style=whitenode] (17) at (2.5, 0.5) {};
		\node [style=whitenode] (18) at (-0.5, -2.5) {};
		\node [style=whitenode] (19) at (-2.5, -0.5) {};
		\node [style=whitenode] (20) at (-2.5, 2.5) {};
		\node [style=whitenode] (21) at (2.5, -2.5) {};
		\node [style=blacknode] (22) at (2.5, -0.5) {};
		\node [style=blacknode] (23) at (0.5, -2.5) {};
		\node [style=blacknode] (24) at (-2.5, 0.5) {};
		\node [style=blacknode] (25) at (-0.5, 2.5) {};
		\node [style=blacknode] (26) at (2.5, 2.5) {};
		\node [style=blacknode] (27) at (-2.5, -2.5) {};
		\node [style=whitenode] (40) at (0.5, 3.5) {};
		\node [style=whitenode] (41) at (3.5, 0.5) {};
		\node [style=whitenode] (42) at (-0.5, -3.5) {};
		\node [style=whitenode] (43) at (-3.5, -0.5) {};
		\node [style=blacknode] (44) at (-3.5, 0.5) {};
		\node [style=blacknode] (45) at (-0.5, 3.5) {};
		\node [style=blacknode] (46) at (3.5, -0.5) {};
		\node [style=blacknode] (47) at (0.5, -3.5) {};
		\node [style=blacknode] (48) at (3.5, 3.5) {};
		\node [style=blacknode] (49) at (-3.5, -3.5) {};
		\node [style=whitenode] (50) at (-3.5, 3.5) {};
		\node [style=whitenode] (51) at (3.5, -3.5) {};
		\node [style=blacknode] (52) at (-0.5, 4.5) {};
		\node [style=blacknode] (53) at (4.5, -0.5) {};
		\node [style=blacknode] (54) at (0.5, -4.5) {};
		\node [style=blacknode] (55) at (-4.5, -4.5) {};
		\node [style=blacknode] (56) at (-4.5, 0.5) {};
		\node [style=whitenode] (57) at (0.5, 4.5) {};
		\node [style=whitenode] (58) at (-4.5, 4.5) {};
		\node [style=whitenode] (59) at (4.5, 0.5) {};
		\node [style=whitenode] (60) at (4.5, -4.5) {};
		\node [style=whitenode] (61) at (-0.5, -4.5) {};
		\node [style=whitenode] (62) at (-4.5, -0.5) {};
		\node [style=blacknode] (63) at (4.5, 4.5) {};

	\end{pgfonlayer}
	\begin{pgfonlayer}{edgelayer}
		\draw [style=blue_thick] (0.center) to (2);
		\draw [style=blue_thick] (2) to (1.center);
		\draw [style=blue_thick] (3) to (1.center);
		\draw [style=blue_thick] (0.center) to (3);
		\draw [style=blue_thick] (6) to (4.center);
		\draw [style=blue_thick] (4.center) to (12);
		\draw [style=blue_thick] (12) to (8.center);
		\draw [style=blue_thick] (8.center) to (11);
		\draw [style=blue_thick] (11) to (14.center);
		\draw [style=blue_thick] (7) to (14.center);
		\draw [style=blue_thick] (5.center) to (7);
		\draw [style=blue_thick] (13) to (5.center);
		\draw [style=blue_thick] (9.center) to (13);
		\draw [style=blue_thick] (10) to (9.center);
		\draw [style=blue_thick] (15.center) to (10);
		\draw [style=blue_thick] (15.center) to (6);
		\draw [style=grayedge] (6) to (0.center);
		\draw [style=grayedge] (4.center) to (2);
		\draw [style=grayedge] (2) to (8.center);
		\draw [style=grayedge] (9.center) to (3);
		\draw [style=grayedge] (10) to (0.center);
		\draw [style=grayedge] (1.center) to (11);
		\draw [style=grayedge] (1.center) to (7);
		\draw [style=grayedge] (3) to (5.center);
		\draw [style=blue_thick] (20.center) to (25);
		\draw [style=blue_thick] (25) to (16.center);
		\draw [style=blue_thick] (16.center) to (26);
		\draw [style=blue_thick] (26) to (17.center);
		\draw [style=blue_thick] (17.center) to (22);
		\draw [style=blue_thick] (22) to (21.center);
		\draw [style=blue_thick] (23) to (21.center);
		\draw [style=blue_thick] (18.center) to (23);
		\draw [style=blue_thick] (27) to (18.center);
		\draw [style=blue_thick] (19.center) to (27);
		\draw [style=blue_thick] (24) to (19.center);
		\draw [style=blue_thick] (20.center) to (24);
		\draw [style=grayedge] (25) to (4.center);
		\draw [style=grayedge] (16.center) to (6);
		\draw [style=grayedge] (15.center) to (25);
		\draw [style=grayedge] (12) to (16.center);
		\draw [style=grayedge] (12) to (17.center);
		\draw [style=grayedge] (8.center) to (22);
		\draw [style=grayedge] (17.center) to (11);
		\draw [style=grayedge] (23) to (5.center);
		\draw [style=grayedge] (7) to (18.center);
		\draw [style=grayedge] (10) to (19.center);
		\draw [style=grayedge] (24) to (9.center);
		\draw [style=grayedge] (15.center) to (24);
		\draw [style=grayedge] (13) to (19.center);
		\draw [style=grayedge] (13) to (18.center);
		\draw [style=grayedge] (14.center) to (22);
		\draw [style=grayedge] (23) to (14.center);
		\draw [style=blue_thick] (50.center) to (45);
		\draw [style=blue_thick] (45) to (40.center);
		\draw [style=blue_thick] (40.center) to (48);
		\draw [style=blue_thick] (48) to (41.center);
		\draw [style=blue_thick] (41.center) to (46);
		\draw [style=blue_thick] (46) to (51.center);
		\draw [style=blue_thick] (47) to (51.center);
		\draw [style=blue_thick] (42.center) to (47);
		\draw [style=blue_thick] (49) to (42.center);
		\draw [style=blue_thick] (43.center) to (49);
		\draw [style=blue_thick] (44) to (43.center);
		\draw [style=blue_thick] (50.center) to (44);
		\draw [style=grayedge] (20.center) to (45);
		\draw [style=grayedge] (44) to (20.center);
		\draw [style=grayedge] (26) to (40.center);
		\draw [style=grayedge] (26) to (41.center);
		\draw [style=grayedge] (27) to (43.center);
		\draw [style=grayedge] (27) to (42.center);
		\draw [style=grayedge] (21.center) to (47);
		\draw [style=grayedge] (21.center) to (46);
		\draw [style=rededge_thick, bend left=15] (20.center) to (26);
		\draw [style=rededge_thick, bend right=15] (20.center) to (27);
		\draw [style=rededge_thick, bend left=15] (26) to (21.center);
		\draw [style=rededge_thick, bend right=15] (27) to (21.center);
		\draw [style=blue_thick] (54) to (60.center);
		\draw [style=blue_thick] (53) to (60.center);
		\draw [style=blue_thick] (59.center) to (53);
		\draw [style=blue_thick] (63) to (59.center);
		\draw [style=blue_thick] (57.center) to (63);
		\draw [style=blue_thick] (52) to (57.center);
		\draw [style=blue_thick] (58.center) to (52);
		\draw [style=blue_thick] (58.center) to (56);
		\draw [style=blue_thick] (56) to (62.center);
		\draw [style=blue_thick] (62.center) to (55);
		\draw [style=blue_thick] (55) to (61.center);
		\draw [style=blue_thick] (61.center) to (54);
		\draw [style=grayedge] (45) to (57.center);
		\draw [style=grayedge] (52) to (40.center);
		\draw [style=grayedge] (41.center) to (53);
		\draw [style=grayedge] (59.center) to (46);
		\draw [style=grayedge] (48) to (57.center);
		\draw [style=grayedge] (48) to (59.center);
		\draw [style=grayedge] (51.center) to (53);
		\draw [style=grayedge] (51.center) to (54);
		\draw [style=grayedge] (42.center) to (54);
		\draw [style=grayedge] (47) to (61.center);
		\draw [style=grayedge] (49) to (62.center);
		\draw [style=grayedge] (49) to (61.center);
		\draw [style=grayedge] (43.center) to (56);
		\draw [style=grayedge] (44) to (62.center);
		\draw [style=grayedge] (50.center) to (56);
		\draw [style=grayedge] (50.center) to (52);
		\draw [style=grayedge, bend left] (58.center) to (63);
		\draw [style=grayedge, bend right] (58.center) to (55);
		\draw [style=grayedge, bend right] (55) to (60.center);
		\draw [style=grayedge, bend right] (60.center) to (63);
	\end{pgfonlayer}
\end{tikzpicture}

  \caption{The graph $G_{2}$}
  \label{G2}
\end{figure}
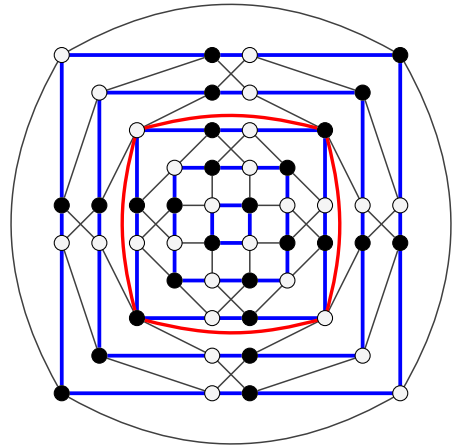

We propose the following conjecture based on the construction in Theorem \ref{4-connected extremal graph}.

\begin{conjecture}
Let $G$ be a  MBICP-graph with $\kappa(G)=4$. 
Then $e(G)\geq\frac{13n(G)}{6}-\frac{14}{3}$.
\end{conjecture}

\section*{Appendix I}\label{appendix}

{\small 
\noindent In the following, we  prove that $\kappa(G_{k})=4$ in Theorem \ref{4-connected extremal graph}. 

Two base cycles $C_{i}$ and $C_{j}$ of $G_{k}$ are called {\it adjacent} if $e(V(C_{i}),V(C_{j}))>0$.
Two subgraphs $H_{1}$ and $H_{2}$ of $G_{k}$ are called {\it linked} if $e(V(H_{1}),V(H_{2}))>0$.
A vertex $u\in V(C_{i})$ is a {\it direct join-vertex} of $G_{k}$ if $e(u,V(C_{j}))>0$,  
where $C_{j}$ is an adjacent base cycle of $C_{i}$;
otherwise, $u$ is an {\it undirect join-vertex} of $G_{k}$.
Observe that in $G_{k}$, only exist four undirect join-vertices, which are the four vertices on the boundary of the outermost $4$-cycle of $G_{k}$,
while all other vertices of $G_{k}$ are direct join-vertices.

\begin{observation}\label{rule}
Let $C_{i}$ and $C_{j}$ be two base cycles of $G_{k}$, if $|i-j|=1$, then,
\begin{itemize}
  \item[(i)] $e(V(C_{i}),V(C_{j}))\geq8$;
  \item[(ii)] for any vertex $u\in  V(C_{i})$, $e(u,V(C_{j}))\leq2$; and
  \item[(iii)] for each $i\in[1,2k+1]$, if $H$ is a component of $C_{i}\setminus S$ with $v(H)\geq3$, then $H$ is linked to all adjacent base cycles of $C_{i}$.
\end{itemize} 
\end{observation}

Given $S\subset V(G)$ with $|S|\leq3$. 
To prove $\kappa(G)=4$, we use the following fact.

\begin{observation}\label{c2k+1}
If $|S|\leq3$, then each component of $G_{k}[V(C_{2k+1})\backslash S]$ contains at least one direct join-vertex.
\end{observation}
\begin{proof}
Since there are four undirect join-vertices and eight direct join-vertices on $C_{2k+1}$. 
Note that each undirect join-vertex has degree four and is adjacent to two distinct direct join-vertices of $V(C_{2k+1})$.
Thus, no component of $G_{k}[V(C_{2k+1})\setminus S]$ consists solely of undirect join-vertices, as this would require $|S|\geq4$.
Therefore, each component $H$ of $G_{k}[V(C_{2k+1})\setminus S]$ contains at least one direct join-vertex such that $e(V(H),V(C_{2k}))>0$.
\end{proof}

Let $t_{i}=|V(C_{i})\cap S|$ for $1\leq i\leq 2k+1$.
Then there exists at least one component $H$ of $C_{i}\backslash S$ such that $v(H)\geq \frac{12-t_{i}}{t_{i}}$.

Now, we will prove $G_{k}$ is $4$-connected by showing the following claims.

\setcounter{claim}{0}
\renewcommand{\thecase}{\arabic{claim}}
\begin{claim}\label{claim1 in theorem 6.1}
If $t_{i}=3$, where $i\in [1,2k+1]$.
Then $G_{k}\backslash S$ is connected.
\end{claim}
\begin{proof}
If $i=1$, then $t_{1}=|V(C_{1})\cap S|=3$.
For any vertex $u\in V(C_{1})$, we have $e(u,V(C_{2}))=2$, then the conclusion holds.

If $i\in [2,2k]$, then $G_{k}[V(C_{i})\backslash S]$ has at most three components,  
with at least one component $H$ satisfying $v(H)\geq\frac{12-3}{3}\geq3$.
By Observation \ref{rule} $(iii)$, $H$ is linked to $C_{i-1}$ and $C_{i+1}$,
i.e., $e(V(H),V(C_{i-1}))>0$ and $e(V(H),V(C_{i+1}))>0$.
Since every vertex on $C_{i}$ is a direct join-vertex, 
each remaining component of $G_{k}[V(C_{i})\setminus S]$ is linked to at least one adjacent base cycle of $C_{i}$.
Thus, graph $G_{k}[(V(C_{i-1})\cup V(C_{i}) \cup V(C_{i+1}))\setminus S]$ is connected.
Furthermore, graphs $G_{k}[V(C_{1})\cup V(C_{2})\cup \cdots \cup V(C_{i-1})]$ and $G_{k}[V(C_{i+1})\cup V(C_{i+2})\cup \cdots \cup V(C_{2k+1})]$ are connected,
implying $G_{k}\backslash S$ is connected.

If $i=2k+1$, by Observation \ref{c2k+1},
each component $H$ of $G_{k}[V(C_{2k+1})\setminus S]$ contains at least one direct join-vertex such that $e(V(H),V(C_{2k}))>0$.
It follows that $G_{k}[(V(C_{2k+1})\cup V(C_{2k}))\setminus S]$ is connected.
Moreover, graph $G_{k}[V(C_{1})\cup V(C_{2})\cup \cdots \cup V(C_{2k})]$ is connected, then $G_{k}\backslash S$ is connected.
Hence Claim \ref{claim1 in theorem 6.1} holds.
\end{proof}

\begin{claim}\label{claim2 in theorem 6.1}
If there exist two integers $p,q\in [1,2k+1]$ such that $t_{p}+t_{q}=3$, where $t_{p}$ and $t_{q}$ are positive integers.
Then $G_{k}\backslash S$ is connected.
\end{claim}
\begin{proof}
If $|p-q|=1$, then $C_{p}$ and $C_{q}$ are two adjacent base cycles of $G_{k}$.
Without loss of generality, assume $t_{p}=1$ and $t_{q}=2$.
If $p\in [1,2k+1]$, then $G_{k}[V(C_{p})\backslash S]$ is connected.
Since $v(C_{p}\backslash S)\geq 3$, by Observation \ref{rule} $(iii)$, 
$G_{k}[V(C_{p})\backslash S]$ is linked to all adjacent base cycles of $C_{p}$ in $G_{k}$.

If $q=1$, then clearly $p=2$. In this case, the two adjacent base cycles are $C_{1}$ and $C_{2}$. 
Since $C_{2}\backslash S$ is a path of order ten, 
by Observation \ref{rule} $(iii)$, $G_{k}[V(C_{2})\backslash S]$ is linked to $C_{1}$ and $C_{3}$.
Thus, $G_{k}[(V(C_{2})\cup V(C_{3})\cup \cdots \cup V(C_{2k+1}))\backslash S]$ is connected.
For any vertex $u$ of $V(C_{1})\backslash S$, since $e(u,V(C_{2}))=2$,  it follows that $e(u,V(C_{2})\setminus S)\geq2-t_{2}=1$, 
confirming that $G_{k}[(V(C_{1})\cup V(C_{2}))\backslash S ]$ is connected.
Hence, $G_{k}\backslash S$ is connected.

If $q\in [2,2k+1]$, then $G_{k}[V(C_{q})\backslash S]$ has at most two components,
with at least one component $H$ satisfying $v(H)\geq\frac{12-2}{2}\geq5$.
By Observation \ref{rule} $(iii)$, $H$ is linked to all adjacent base cycles of $C_{q}$.
Since every vertex on $C_{i}$ $(i\in[1,2k])$ is a direct join-vertex, 
and by Observation \ref{c2k+1}, each component of $G_{k}[V(C_{2k+1})\setminus S]$ contains at least one direct join-vertex when $|S|\leq3$, 
the remaining component of $G_{k}[V(C_{q})\backslash S]$ is linked to at least one adjacent base cycle of $C_{q}$.
Additionally, by Observation \ref{rule} $(i)$ and $(ii)$, $e(V(C_{p})\backslash S,V(C_{q})\backslash S)\geq 8-2t_{p}-2t_{q}\geq2$.
It implies $G_{k}\backslash S$ is connected.

If $|p-q|\geq2$, then $C_{p}$ and $C_{q}$ are non-adjacent base cycles of $G_{k}$.
Clearly $G_{k}\backslash S$ is connected,
the analysis follows analogously to the case $|p-q|=1$. 
Thus, Claim \ref{claim2 in theorem 6.1} holds.
\end{proof}

\begin{claim}\label{claim3 in theorem 6.1}
If there exist three integers $p,q,r\in [1,2k+1]$ such that $t_{p}=t_{q}=t_{r}=1$.
Then $G_{k}\backslash S$ is connected.
\end{claim}
\begin{proof}
Since $t_{p}=t_{q}=t_{r}=1$, $G_{k}[V(C_{p}\setminus S]$, $G_{k}[V(C_{q}\setminus S]$ and $G_{k}[V(C_{r}\setminus S]$ are connected graphs of order at least three.
Suppose $C_{p}$, $C_{q}$ and $C_{r}$ are three non-adjacent base cycles of $G_{k}$.
By Observation \ref{rule} $(iii)$, $G_{k}[V(C_{p}\setminus S]$, $G_{k}[V(C_{q}\setminus S]$ and $G_{k}[V(C_{r}\setminus S]$ are each connected to all their adjacent base cycles in $G_{k}$.
Thus, $G_{k}\backslash S$ is clearly connected.

Now, assume $C_{p}$, $C_{q}$ and $C_{r}$ are three consecutive adjacent base cycles of $G_{k}$.
By Observation \ref{rule} $(i)$ and $(ii)$, $e(V(C_{p})\setminus S, V(C_{q})\setminus S)\geq 8-2t_{p}-2t_{q}\geq 8-2-2=4$
and $e(V(C_{q})\setminus S, V(C_{r})\setminus S)\geq 8-2t_{q}-2t_{r}\geq 8-2-2=4$.
Therefore, $G_{k}[(V(C_{p})\cup V(C_{q}) \cup V(C_{r}))\setminus S]$ is connected.
Since $v(C_{p}\setminus S)\geq 3$ and $v(C_{r}\setminus S)\geq 3$, 
by Observation \ref{rule} $(iii)$, $G_{k}[V(C_{p}\setminus S]$ and $G_{k}[V(C_{r}\setminus S]$ are each connected to all their adjacent base cycles in $G_{k}$.
Hence, $G_{k}\backslash S$ is connected.

Suppose that two of $C_{p}$, $C_{q}$ and $C_{r}$ are non-adjacent in $G_{k}$, 
then clearly $G_{k}\backslash S$ is connected,
the proof follows similarly to the above argument.
Claim \ref{claim3 in theorem 6.1} holds.
\end{proof}

Hence, for any subset $S$ with $|S|\leq3$, $G_{k}\backslash S$ is connected,
Thus $G_{k}$ is $4$-connected. 
}

\end{document}